\documentclass[a4paper,11pt]{amsart}

\usepackage[top=30truemm,bottom=25truemm,left=35truemm,right=35truemm]{geometry}

\usepackage{amsmath}
\usepackage{amssymb}
\usepackage{amsthm}
\usepackage{graphics}
\usepackage[abbrev,alphabetic]{amsrefs}
\usepackage{amscd}
\usepackage[dvipdfmx]{graphicx}
\usepackage{color}
\usepackage{ulem}

\usepackage[all,cmtip]{xy}
\usepackage{xypic}

\title[Inversion of adjunction for quotient singularities III]
{Inversion of adjunction for quotient singularities III: semi-invariant case}

\author{Yusuke Nakamura}
\address{Graduate School of Mathematics, Nagoya University, Furo-cho, Chikusa-ku, Nagoya, 464-8602, Japan.}
\email{y.nakamura@math.nagoya-u.ac.jp}
\urladdr{https://sites.google.com/site/ynakamuraagmath/}

\author{Kohsuke Shibata}
\address{School of Engineering, Tokyo Denki University, Adachi-ku, Tokyo 120-8551, Japan.}
\email{shibata.kohsuke@mail.dendai.ac.jp}

\subjclass[2020]{Primary 14E18; Secondary 14E30, 14B05}

\keywords{minimal log discrepancy, arc space, hyperquotient singularity, LSC conjecture, PIA conjecture}

\newtheorem{thm}{Theorem}[section]
\newtheorem{lem}[thm]{Lemma}

\newtheorem{cor}[thm]{Corollary}
\newtheorem{prop}[thm]{Proposition}

\theoremstyle{definition}
\newtheorem{defi}[thm]{Definition}

\newtheorem{conj}[thm]{Conjecture}

\theoremstyle{remark}
\newtheorem{rmk}[thm]{Remark}

\newtheorem*{ackn}{Acknowledgements}

\begin{document}

\maketitle

\begin{abstract}
We prove the precise inversion of adjunction formula for finite linear group quotients of complete intersection varieties defined by semi-invariant equations. 
As an application, we prove the semi-continuity of minimal log discrepancies for them. 
These results extend the results in our first paper, where we prove the same results for finite linear group quotients of complete intersection varieties defined by ``invariant equations". 
\end{abstract}

\section{Introduction}
The minimal log discrepancy is an important invariant of singularities in birational geometry.
In \cite{Amb99}, Ambro proposed the LSC (lower semi-continuity) conjecture for minimal log discrepancies.
It is known that the LSC conjecture and the ACC (ascending chain condition) conjecture imply the conjecture of termination of flips \cite{Sho04}.
In this paper, we always work over an algebraic closed field $k$ of characteristic zero. 

\begin{conj}[LSC conjecture]\label{conj:LSC}
Let $(X, \mathfrak{a})$ be a log pair, and let $|X|$ denote the set of all closed points of $X$ with the Zariski topology. 
Then the function 
\[
|X| \to \mathbb{R}_{\ge 0} \cup \{ - \infty \}; \quad x \mapsto \operatorname{mld}_x (X,\mathfrak{a})
\]
is lower semi-continuous. 
\end{conj}

In \cite{Amb99}, Ambro proves the LSC conjecture when $\dim X = 3$.
In \cite{EMY03}, Ein, Musta\c{t}\u{a} and Yasuda prove it when $X$ is smooth. In \cite{EM04}, Ein and Musta\c{t}\u{a} generalize it to normal local complete intersection varieties. 
In \cite{Nak16}, the first author proves it for quotient singularities.
In \cite{NS22}, the authors prove the conjecture for a finite linear group quotient of a complete intersection variety defined by invariant equations.
In \cite{NS2}, the authors generalize it to a non-linear group action.

The main purpose of this paper is to relax the assumptions of ``invariant equations" in the results in \cite{NS22} and extend them to the case of ``semi-invariant equations". 
For a finite subgroup $G\subset {\rm GL}_N(k)$, we say that $f \in k[x_1, \ldots , x_N]$ is \textit{$G$-semi-invariant} if for any $\gamma \in G$, there exists $s \in k^{\times}$ such that $\gamma (f) = s f$.
In this paper, we show that the LSC conjecture is true for a finite linear group quotient of a complete intersection variety defined by semi-invariant equations. 

The advantage of being able to handle semi-invariant equations is that it enables us to treat a class containing 3-dimensional terminal singularities. 
According to the classification by Reid and Mori (\cites{Rei80, Rei81, Mor85}), 
3-dimensional terminal singularities are known to be finite quotients of smooth points or cDV points. 
As seen in Mori's classification list (\cite{Mor85}), semi-invariant equations indeed appear alongside invariant equations.

\begin{thm}[{$=$ Theorem \ref{thm:LSC_general}}]\label{thm:LSC_intro}
Let $G \subset {\rm GL}_N(k)$ be a finite subgroup which does not contain a pseudo-reflection. 
Let $X := \mathbb{A}_k^N / G$ be the quotient variety. 
Let $Z \subset X$ be the minimum closed subset such that $\mathbb{A}_k^N \to X$ is \'{e}tale outside $Z$. 
Let $\overline{Y}$ be a subvariety of $\mathbb{A}_k^N$ of codimension $c$. 
Suppose that $\overline{Y} \subset \mathbb{A}_k^N$ is defined by $c$ $G$-semi-invariant equations 
$f_1, \ldots , f_c \in k[x_1, \ldots, x_N]$. 
Let $Y := \overline{Y}/ G$ be the quotient variety. 
We assume that $Y$ has only klt singularities and $\operatorname{codim} _Y (Y \cap Z) \ge 2$. 
Then, for any $\mathbb{R}$-ideal sheaf $\mathfrak{b}$ on $Y$, the function 
\[
|Y| \to \mathbb{R}_{\ge 0} \cup \{ - \infty \}; \quad y \mapsto \operatorname{mld}_y(Y,\mathfrak{b})
\]
is lower semi-continuous, where we denote by $|Y|$ the set of all closed points of $Y$ with the Zariski topology. 
\end{thm}

Motivated by the reduction of Theorem \ref{thm:LSC_intro} to the case of quotient singularities, 
this paper also investigates the PIA conjecture for quotient singularities.

\begin{conj}[PIA conjecture, {\cite{92}*{17.3.1}}]\label{conj:PIA}
Let $(X, \mathfrak{a})$ be a log pair, and let $D$  be a normal Cartier prime divisor.
Let $x \in D$ be a closed point. Suppose that $D$ is not contained in the cosupport of the $\mathbb{R}$-ideal sheaf $\mathfrak{a}$. 
Then 
\[
\operatorname{mld}_x \bigl( X,D, \mathfrak{a} \mathcal{O}_X \bigr) = \operatorname{mld}_x (D, \mathfrak{a} \mathcal{O}_D)
\]
holds. 
\end{conj}

In \cite{EMY03}, Ein, Musta\c{t}\u{a} and Yasuda prove the PIA conjecture for a smooth variety. 
In \cite{EM04}, Ein and Musta\c{t}\u{a} generalize it to a normal local complete intersection variety. 
In \cite{NS22}, the authors prove the conjecture for a finite linear group quotient of a complete intersection variety defined by invariant equations.
In \cite{NS2}, the authors generalize it to a non-linear group action.

In this paper, we show that the PIA conjecture holds for a finite linear group quotient of a complete intersection variety defined by semi-invariant equations.

\begin{thm}[{$=$ Theorem \ref{thm:PIA3}}]\label{thm:PIA_intro}
Let $G \subset {\rm GL}_N(k)$ be a finite subgroup which does not contain a pseudo-reflection. 
Let $X := \mathbb{A}_k^N / G$ be the quotient variety. 
Let $Z \subset X$ be the minimum closed subset such that $\mathbb{A}_k^N \to X$ is \'{e}tale outside $Z$. 
Let $\overline{Y}$ be a subvariety of $\mathbb{A}_k^N$ of codimension $c$. 
Suppose that $\overline{Y} \subset \mathbb{A}_k^N$ is defined by $c$ $G$-semi-invariant equations 
$f_1, \ldots , f_c \in k[x_1, \ldots, x_N]$. 
Let $Y := \overline{Y}/ G$ be the quotient variety. 
We assume that $Y$ has only klt singularities and $\operatorname{codim} _Y (Y \cap Z) \ge 2$. 
Suppose that $\operatorname{codim}_D (D \cap Z) \ge 2$ and that $D$ is klt at a closed point $y \in D$. 
Then, for any $\mathbb{R}$-ideal sheaf $\mathfrak{a}$ on $Y$ with $\mathfrak{a} \mathcal{O}_D \not = 0$, we have 
\[
\operatorname{mld}_y \left( Y,D, \mathfrak{a} \right) = 
\operatorname{mld}_y \left( D, \mathfrak{a} \mathcal{O}_D \right). 
\]
\end{thm}

In order to reduce Theorem \ref{thm:LSC_intro} to the case of quotient singularities, we prove the following form of the PIA conjecture.

\begin{thm}[{$=$ Theorem \ref{thm:PIA2}}]\label{thm:PIA2_intro}
Let $G \subset {\rm GL}_N(k)$ be a finite subgroup of order $d$  which does not contain a pseudo-reflection. 
Let $X := \mathbb{A}_k^N / G$ be the quotient variety. 
Let $Z \subset X$ be the minimum closed subset such that $\mathbb{A}_k^N \to X$ is \'{e}tale outside $Z$. 
Let $\overline{Y}$ be a subvariety of $\mathbb{A}_k^N$ of codimension $c$. 
Suppose that $\overline{Y} \subset \mathbb{A}_k^N$ is defined by $c$ $G$-semi-invariant equations 
$f_1, \ldots , f_c \in k[x_1, \ldots, x_N]$. 
Let $Y := \overline{Y}/ G$ be the quotient variety. 
We assume that $Y$ has only klt singularities and $\operatorname{codim} _Y (Y \cap Z) \ge 2$. 
Then, for any $\mathbb{R}$-ideal sheaf $\mathfrak{a}$ on $X$ 
with $\mathfrak{a} \mathcal{O}_Y \not = 0$, and any closed point $y \in Y$, we have 
\[
\operatorname{mld}_y \bigl( Y, \mathfrak{a}\mathcal{O}_Y \bigr) = 
\operatorname{mld}_y \bigl( X, (f_1 ^d  \cdots f_{c} ^d )^{\frac{1}{d}} \mathfrak{a} \bigr). 
\]
\end{thm}

This paper is mainly devoted to proving Theorem \ref{thm:PIA2_intro}. 
The proof of Theorem \ref{thm:PIA2_intro} is given using the theory of arc spaces as in \cite{NS22}. 
However, since we are dealing with semi-invariant equations instead of invariant equations, some modifications are required in the proof. 
Below, we will explain the difficulties in dealing with semi-invariant equations and the necessary modifications. 

Let $N$ be a positive integer. Set $R := k[x_1, \ldots , x_N]$. 
Let $d$ be a positive integer, and let $\xi \in k$ be a primitive $d$-th root of unity. 
Suppose that a finite group $G \subset \operatorname{GL}_N(k)$ of order $d$ acts on $\overline{A} = \operatorname{Spec} R$. 
Let $\overline{B} \subset \overline{A}$ be a complete intersection variety defined by $G$-semi-invariant elements $f_1,\ldots,f_c \in (x_1, \ldots, x_N) \subset R$. 
We denote by $B := \overline{B}/G$ its quotient variety. 
Since $G$ is a finite group, each $\gamma$ can be diagonalized with respect to a new basis $x_1^{(\gamma)}, \ldots, x_N^{(\gamma)}$ consisting of $k$-linear combinations of the original basis $x_1, \ldots , x_N$.
Let $\operatorname{diag} \left( \xi ^{e_1}, \ldots , \xi ^{e_N} \right)$ be the diagonal matrix 
with $0 \le e_i \le d-1$. 
Then we define a $k$-algebra homomorphism ${\lambda}^* _{\gamma}:R \to R[t^{1/d}]$ as follows:
\[
{\lambda}^* _{\gamma}: R \to R[t^{1/d}]; 
\quad x_i^{(\gamma)} \mapsto t^{e_i/d}x_i^{(\gamma)}. 
\]

When $f_i$ is $G$-invariant, we have $\lambda ^* _{\gamma}(f_i) \in R[t]$. 
In \cite{NS22}, we deal with this case, and we define a $k[t]$-scheme
\[\tag{1}
\overline{B}^{(\gamma)} 
:= \operatorname{Spec} \left( R[t]/\left( \lambda ^* _{\gamma}(f_1), \ldots , \lambda ^* _{\gamma}(f_c) \right) \right ). 
\]
In \cite{NS22}, we study the minimal log discrepancies of $B$ via the arc space $\overline{B}^{(\gamma)} _{\infty}$ by applying Denef and Loeser's theory \cite{DL02}.
Here, $\overline{B}^{(\gamma)} _{\infty}$ is the arc space of $\overline{B}^{(\gamma)}$ as a $k[t]$-scheme (Subsection \ref{subsection:arc}), which is essentially the corresponding Greenberg scheme (see Remark \ref{rmk:GB} for details).

However, if $f_i$ is just assumed to be $G$-semi-invariant, $\lambda ^* _{\gamma}(f_i)$ is not necessarily an element of $R[t]$.
Thus, the first difficulty is how to redefine $\overline{B}^{(\gamma)}$. 
For each $i$ and $\gamma \in G$, we define $w_\gamma(f_i):=\frac{a}{d}$ when $\gamma (f_i) =\xi^a f_i$ with $0 \le a \le d-1$. 
Then, we have $\lambda ^* _{\gamma}(f_i)t^{-w_\gamma(f_i)} \in R[t]$. 
By using this invariant $w_{\gamma}$, we change the definition of $\overline{B}^{(\gamma)}$ as follows: 
\[
\overline{B}^{(\gamma)} 
:= \operatorname{Spec} \left(R[t]/ \bigl( \lambda ^* _{\gamma}(f_1)t^{-w_\gamma(f_1)}, \ldots , \lambda ^* _{\gamma}(f_c)t^{-w_\gamma(f_c)} \bigr)\right).
\]
This definition coincides with the original definition (1) when $f_i$ is $G$-invariant, 
since we have $w_\gamma(f_i) = 0$ in this case.

As in \cite{NS22}, we can describe the minimal log discrepancies of $B$ by the codimensions of certain contact loci in the arc space $\overline{B}^{(\gamma)}_{\infty}$ of $\overline{B}^{(\gamma)}$. 
This is a crucial key to proving Theorem \ref{thm:PIA2_intro}. 

\begin{thm}[{$=$ Theorem \ref{thm:mld_hyperquot}}]\label{thm:mld_hyperquot_intro}
Let $n := N - c$. 
Let $x \in B$ be the image of the origin of $\overline{A} = \mathbb{A}_k ^N$. 
Let $\mathfrak{a} \subset  \mathcal{O}_B$ be a non-zero ideal sheaf and $\tau$ a positive real number. 
Then 
\begin{align*}
\operatorname{mld}_x(B,\mathfrak{a}^{\tau}) 
&= \inf _{\gamma \in G, \ b_1, b_2 \in \mathbb{Z}_{\ge 0}}
\left \{ 
\operatorname{codim} \left( C_{\gamma, b_1, b_2} \right) + \operatorname{age}(\gamma)- w_\gamma({\bf f}) - b_2 - \tau b_1
\right \} \\
&= \inf _{\gamma \in G, \ b_1, b_2 \in \mathbb{Z}_{\ge 0}}
\left \{ 
\operatorname{codim} \left( C' _{\gamma, b_1, b_2} \right) + \operatorname{age}(\gamma) - w_\gamma({\bf f}) - b_2 - \tau b_1
\right \}
\end{align*}
holds for 
\begin{align*}
C_{\gamma, b_1, b_2} 
& := 
\operatorname{Cont}^{\ge 1} \left( \mathfrak{m}_x \mathcal{O}_{\overline{B}^{(\gamma)}} \right) \cap
\operatorname{Cont}^{b_1}  \left( \mathfrak{a} \mathcal{O}_{\overline{B}^{(\gamma)}} \right) \cap  
\operatorname{Cont}^{b_2} \left(\operatorname{Jac}_{\overline{B}^{(\gamma)}/k[t]} \right), \\
C'_{\gamma, b_1, b_2}
& := 
\operatorname{Cont}^{\ge 1} \left(\mathfrak{m}_x \mathcal{O}_{\overline{B}^{(\gamma)}} \right) \cap
\operatorname{Cont}^{\ge b_1}  \left(\mathfrak{a} \mathcal{O}_{\overline{B}^{(\gamma)}} \right) \cap 
\operatorname{Cont}^{b_2} \left(\operatorname{Jac}_{\overline{B}^{(\gamma)}/k[t]} \right),
\end{align*}
where $\operatorname{Jac}_{\overline{B}^{(\gamma)} /k[t]} := 
\operatorname{Fitt}^{n} \left( \Omega _{\overline{B}^{(\gamma)} /k[t]} \right)$ and $w_\gamma({\bf f}) := w_{\gamma}(f_1) + \cdots + w_{\gamma}(f_c)$.
\end{thm}

This result is a generalization of Theorem 4.8 and Corollary 4.9 in \cite{NS22}. 
Compared to the theorems in \cite{NS22}, the new term $w_\gamma({\bf f})$ appears in Theorem \ref{thm:mld_hyperquot_intro}.

By the theorem of Ein, Musta{\c{t}}{\u{a}}, and Yasuda (\cite{EMY03}), 
the minimal log discrepancy $\operatorname{mld}_x(B,\mathfrak{a}^{\tau})$ is described by the codimensions of the corresponding contact loci in the arc space $B_{\infty}$ of $B$. 
Therefore, in order to prove Theorem \ref{thm:mld_hyperquot_intro}, 
we need to compare the codimension of a cylinder $C \subset B_{\infty}$ and the codimension of 
the corresponding cylinder $C' \subset \overline{B}^{(\gamma)}_{\infty}$. 
By the same argument in \cite{NS22} (using the theory of Denef and Loeser in \cite{DL02}), the difference can be measured by comparing $\Omega _{\overline{B}^{(\gamma)} / k[t]} ^n$ and $\omega _{B}$ via the morphism $\overline{B}^{(\gamma)} \to B$. 

In \cite{NS22}, in order to study the sheaf $\Omega _{\overline{B}^{(\gamma)} / k[t]}$ of differentials, 
we define an invertible sheaf $L_{\overline{B}^{(\gamma)}}$ on $\overline{B}^{(\gamma)}$ as 
\[
L_{\overline{B}^{(\gamma)}} 
:= 
\bigl( \operatorname{det}^{-1} (I_B/I_B^2) \bigr) \big|_{\overline{B}^{(\gamma)}}
\otimes _{\mathcal{O}_{\overline{B}^{(\gamma)}}} 
\Bigl( \Omega _{\overline{A}^{(\gamma)} /k[t]} ^N \Bigr) \big|_{\overline{B}^{(\gamma)}}, 
\]
where $I_B := \left( f_1, \ldots , f_c \right) \subset R^G$ and 
$\overline{A}^{(\gamma)} := \operatorname{Spec} R[t]$. 
In \cite{NS22}, this $L_{\overline{B}^{(\gamma)}}$ plays the similar role as the relative canonical sheaf $\omega _{\overline{B}^{(\gamma)} / k[t]}$, and therefore, $\Omega _{\overline{B}^{(\gamma)} / k[t]}$ can be compared with $\omega _{B}$ via the new sheaf $L_{\overline{B}^{(\gamma)}}$. 
Note that we have no standard definition of $\omega _{\overline{B}^{(\gamma)} / k[t]}$ since $\overline{B}^{(\gamma)}$ is neither normal nor a complete intersection in general (see \cite{NS22}*{Remark 4.4}). 
This elaborate definition of $L_{\overline{B}^{(\gamma)}}$ was also an important argument in \cite{NS22}.

When dealing with $G$-semi-invariant equations $f_1, \ldots, f_c$, the same definition of $L_{\overline{B}^{(\gamma)}}$ above does not work because $B$ is not defined by a regular sequence of $R^G$. Here, another difficulty arises as to how $L_{\overline{B}^{(\gamma)}}$ should be redefined.
We then make the following modifications. 
First, note that the ring homomorphism $\lambda ^*_{\gamma}$ induces a ring homomorphism
\[
R/(f_1, \ldots, f_c) \longrightarrow 
R[t^{1/d}]/\bigl( \lambda ^* _{\gamma}(f_1)t^{-w_\gamma(f_1)}, \ldots , \lambda ^* _{\gamma}(f_c)t^{-w_\gamma(f_c)} \bigr). 
\]
Therefore, we have a corresponding morphism
\[
\overline{B}^{(\gamma)}_{1/d} := \overline{B}^{(\gamma)}\times_{\operatorname{Spec} k[t]} \operatorname{Spec} k[t^{1/d}] 
\longrightarrow \overline{B}. 
\]
Then, we define a (not necessarily an invertible) sheaf $L_{\overline{B}^{(\gamma)}_{1/d}}$ on $\overline{B}^{(\gamma)}_{1/d}$ by 
\[
L_{\overline{B}^{(\gamma)}_{1/d}} := 
t^{w_{\gamma}({\bf f})}
\bigl( \operatorname{det}^{-1} (I_{\overline{B}}/I_{\overline{B}}^2) \bigr) \big|_{\overline{B}^{(\gamma)}_{1/d}}
\otimes _{\overline{B}^{(\gamma)}_{1/d}} 
\left( \Omega _{\overline{A}^{(\gamma)}_{1/d} /k[t]} ^N \right) \big|_{\overline{B}^{(\gamma)}_{1/d}}, 
\]
where $I_{\overline{B}} := \left( f_1, \ldots , f_c \right) \subset R$, 
and $\overline{A}^{(\gamma)}_{1/d} := \operatorname{Spec} R[t^{1/d}]$. 
This $L_{\overline{B}^{(\gamma)}_{1/d}}$ will play the similar role as the $L_{\overline{B}^{(\gamma)}}$ in \cite{NS22}, and it allows us to modify the argument. 
A new point of this paper is that we extend the $k[t]$-schemes appearing in \cite{NS22} to $k[t^{1/d}]$-schemes in this way, and we compare the sheaves of differentials on them. 

The paper is organized as follows. 
In Section \ref{section:pre}, following \cite{NS22}, we review some definitions and facts on pairs and arc spaces. 
In Section \ref{section:setting}, we review the theory of arc spaces of quotient varieties established by Denef and Loeser in \cite{DL02}, and we apply it to our setting. 
In Section \ref{section:order}, we study the orders of the Jacobian for the morphisms $\overline{\mu} _{\gamma}$ and $p$, and we prove some relations using some ideals on $\overline{B}^{(\gamma)}$. 
These relations are the key ingredient of the proof of Theorem \ref{thm:mld_hyperquot}.  
In Section \ref{section:mld}, we prove Theorem \ref{thm:mld_hyperquot}.
In Sections \ref{section:PIA} and \ref{section:LSC}, we prove Theorems \ref{thm:PIA_intro},  \ref{thm:PIA2_intro} and \ref{thm:LSC_intro}.

The klt assumption on $Y$ and $D$ in Theorem \ref{thm:PIA_intro} is used to prove $b_3 < \infty$ in the proof of Theorem \ref{thm:PIA}. 
We cite \cite{NS22}*{Claim 5.2} for this step, which is essentially based on the following two results: the result by Hacon and McKernan \cite{HM07}, which establishes the rational chain connectedness of fibers of resolutions of klt singularities; and the result by Graber, Harris, and Starr \cite{GHS03}, which asserts the existence of a section for a morphism with rational chain connected fibers. 
We emphasize that this assumption cannot be dropped, as \cite{NS4} constructs a counterexample in the non-klt setting. 

\begin{ackn} 
The first author is partially supported by JSPS KAKENHI No.\
18K13384, 22K13888, and JPJSBP120219935. The second author is
partially supported by JSPS KAKENHI No.\ 19K14496 and 23K12958.
\end{ackn}

\section*{Notation}

\begin{itemize}
\item 
We basically follow the notations and the terminologies in \cite{Har77} and \cite{Kol13}.

\item 
Throughout this paper, $k$ is an algebraically closed field  of characteristic zero. 
We say that $X$ is a \textit{variety over} $k$ or a \textit{$k$-variety} if 
$X$ is an integral scheme that is separated and of finite type over $k$. 
\end{itemize}

\section{Preliminaries}\label{section:pre}
Following \cite{NS22}, we review some definitions and facts on pairs and arc spaces.
\subsection{Log pairs}
A \textit{log pair} $(X, \mathfrak{a})$ is a normal $\mathbb{Q}$-Gorenstein $k$-variety $X$ and 
an $\mathbb{R}$-ideal sheaf $\mathfrak{a}$ on $X$. 
Here, an $\mathbb{R}$-\textit{ideal sheaf} $\mathfrak{a}$ on $X$ is a formal product 
$\mathfrak{a} = \prod _{i = 1} ^s \mathfrak{a}_i ^{\tau_i}$, where $\mathfrak{a}_1, \ldots, \mathfrak{a}_s$ are 
non-zero coherent ideal sheaves on $X$ 
and $\tau _1, \ldots , \tau _s$ are positive real numbers. 
For a morphism $Y \to X$ and an $\mathbb{R}$-ideal sheaf $\mathfrak{a} = \prod _{i = 1} ^s \mathfrak{a}_i ^{\tau _i}$ on $X$, 
we denote by $\mathfrak{a} \mathcal{O}_Y$ the $\mathbb{R}$-ideal sheaf $\prod _{i = 1} ^s (\mathfrak{a}_i \mathcal{O}_Y)  ^{\tau _i}$ on $Y$. 

Let $\bigl( X, \mathfrak{a} = \prod _{i = 1} ^s \mathfrak{a}_i ^{\tau _i} \bigr)$ be a log pair. 
Let $f: X' \to X$ be a proper birational morphism from a normal variety $X'$ and let $E$ be a prime divisor on $X'$. 
We denote by $K_{X'/X} := K_{X'} - f^* K_X$ the relative canonical divisor. 
Then the \textit{log discrepancy} of $(X, \mathfrak{a})$ at $E$ is defined as 
\[
a_E(X, \mathfrak{a}) := 1 + \operatorname{ord}_E (K_{X'/X}) - \operatorname{ord}_E \mathfrak{a}, 
\]
where we define $\operatorname{ord}_E \mathfrak{a} := \sum _{i=1} ^s \tau _i \operatorname{ord}_E \mathfrak{a}_i$. 
The image $f(E)$ is called the \textit{center} of $E$ on $X$ and we denote it by $c_X(E)$. 
For a closed point $x \in X$, we define \textit{the minimal log discrepancy} at $x$ as 
\[
\operatorname{mld}_x (X, \mathfrak{a}) := \inf _{c_X(E) = \{ x \}} a_E (X, \mathfrak{a})
\]
if $\dim X \ge 2$, where the infimum is taken over all prime divisors $E$ over $X$ with center $c_X(E) = \{ x \}$. 
It is known that $\operatorname{mld}_x (X, \mathfrak{a}) \in \mathbb{R}_{\ge 0} \cup \{ - \infty \}$ in this case (cf.\ \cite{KM98}*{Corollary 2.31}). 
When $\dim X = 1$, we define $\operatorname{mld}_x (X, \mathfrak{a}) := \inf _{c_X(E) = \{ x \}} a_E (X, \mathfrak{a})$ 
if the infimum is non-negative, and we define $\operatorname{mld}_x (X, \mathfrak{a}) := - \infty$ otherwise. 

\subsection{Log triples}
A \textit{log triple} $(X, D, \mathfrak{a})$ is a normal $\mathbb{Q}$-Gorenstein $k$-variety $X$, a $\mathbb{Q}$-Cartier $\mathbb{Q}$-divisor $D$ on $X$, and 
an $\mathbb{R}$-ideal sheaf $\mathfrak{a}$ on $X$. 

Let $\bigl( X, D, \mathfrak{a} = \prod _{i = 1} ^s \mathfrak{a}_i ^{\tau _i} \bigr)$ be a log triple. 
Let $f: X' \to X$ be a proper birational morphism from a normal variety $X'$ and let $E$ be a prime divisor on $X'$. 
Then the \textit{log discrepancy} of $(X, D, \mathfrak{a})$ at $E$ is defined as 
\[
a_E(X, D, \mathfrak{a}) := 1 + \operatorname{ord}_E (K_{X'}-f^*(K_X+D)) - \operatorname{ord}_E \mathfrak{a}.
\]
Then $\operatorname{mld}_x (X, D, \mathfrak{a})$ is defined in the same way as for log pairs.

\subsection{Jacobian ideals and Nash ideals}
In this subsection, we recall the definition of the Jacobian ideals and the Nash ideals. 
\begin{defi}\label{defi:nash}
\begin{enumerate}
\item  
For a scheme $X$ over $k$ of dimension $n$, we denote by 
$\operatorname{Jac}_X := \operatorname{Fitt} ^n (\Omega _X)$ the \textit{Jacobian ideal} of $X$ (see \cite{Eis95} for the definition of the Fitting ideal). 

\item 
Let $X$ be a normal $\mathbb{Q}$-Gorenstein variety over $k$ of dimension $n$, and 
let $r$ be a positive integer such that the reflexive power $\omega _X ^{[r]} := \left( \omega _X ^{\otimes r} \right)^{**}$ is an invertible sheaf. 
Then we have a canonical map
\[
\eta_r \colon \left( \Omega_X^n \right )^{\otimes r} \to \omega _X ^{[r]}.
\] 
Since $\omega _X ^{[r]}$ is an invertible sheaf, 
an ideal sheaf $\mathfrak{n}_{r,X} \subset \mathcal{O}_X$ is uniquely determined by 
${\rm Im}(\eta_r) = \mathfrak{n}_{r,X} \otimes \omega _X ^{[r]}$.
The ideal sheaf $\mathfrak n_{r,X}$ is called the $r$-th \textit{Nash ideal} of $X$.

\item 
Let $X$ be a normal $k[t]$-variety of relative dimension $n$. Suppose that $X$ is smooth over $k[t]$ outside a closed subset of $X$ of codimension two. 
Then the canonical sheaf $\omega _{X/k[t]}$ is defined (cf.\ \cite{Kol13}*{Definition 1.6}). 
\end{enumerate}
\end{defi}

\begin{rmk}\label{rmk:nash}
In this paper, we use the notation $\omega _{X/k[t]}$ only for $k[t]$-schemes $X$ of the form 
$X = X' \times _{\operatorname{Spec} k} \operatorname{Spec} k[t]$ with some normal $k$-varieties $X'$. 
In this case, we simply have $\omega _{X/k[t]} \simeq \omega _{X'} \otimes _{\mathcal{O}_{X'}} \mathcal{O}_X$. 
\end{rmk}

\subsection{Arc spaces of $k$-schemes}\label{subsection:arc}
In this subsection, we briefly review the definition and some properties of jet schemes and arc spaces. The reader is referred to \cite{EM09} for details.

Let $X$ be a scheme of finite type over $k$. 
Let $({\sf Sch}/k)$ be the category of $k$-schemes and $({\sf Sets})$ the category of sets.
Define a contravariant functor  $F_{m}: ({\sf Sch}/k) \to ({\sf Sets})$ by 
\[
 F_{m}(Y) = \operatorname{Hom} _{k}\left( Y\times_{\operatorname{Spec} k} \operatorname{Spec} k[t]/(t^{m+1}), X \right).
\]
Then, the functor $F_{m}$ is representable by a scheme $X_m$ of finite type over $k$, and 
the scheme $X_m$ is called the $m$-th \textit{jet scheme} of $X$.
For $m \ge n \ge 0$, the canonical surjective homomorphism $k[t]/(t^{m+1}) \to k[t]/(t^{n+1})$ induces a morphism $\pi_{mn}:X_m \to X_n$.
There exists the projective limit and projections
\[
X_\infty := \mathop{\varprojlim}\limits_{m} X_m, \qquad \psi_{m}:X_{\infty} \to X_m, 
\]
and $X_{\infty}$ is called the {\it arc space} of $X$. 
Then there is a bijective map
\[
\operatorname{Hom} _{k}(\operatorname{Spec} K, X_{\infty}) \simeq \operatorname{Hom} _{k}(\operatorname{Spec} K[[t]], X)
\]
for any field $K$ with $k\subset K$. 

For $m\in\mathbb Z_{\ge 0}\cup\{\infty\}$ and a morphism $f:Y\to X$ of  schemes  of  finite type over $k$,
we denote by $f_m : Y_m \to X_m$ the morphism induced by $f$.

A subset $C \subset X_{\infty}$ is called a \textit{cylinder} if $C = \psi_{m} ^{-1}(S)$ holds for some $m \ge 0$ and 
a constructible subset $S \subset X_m$. Typical examples of cylinders appearing in this paper are the \textit{contact loci} 
$\operatorname{Cont}^{m}(\mathfrak{a})$ and $\operatorname{Cont}^{\geq m}(\mathfrak{a})$ defined as follows. 

\begin{defi}\label{defi:ord_cont}
\begin{enumerate}
\item
For an arc $\alpha \in X_{\infty}$ represented by a morphism $\alpha: \operatorname{Spec} K[[t]] \to X$ (where $K$ is a field extension of $k$) and an ideal sheaf $\mathfrak{a} \subset \mathcal{O}_X$, 
the \textit{order} of $\mathfrak{a}$ measured by $\alpha$ is defined as follows:
\[
\operatorname{ord}_{\alpha} (\mathfrak{a}) = \sup \{ r \in \mathbb{Z}_{\geq 0} \mid \alpha^*(\mathfrak{a}) \subset (t^r) \} \in \mathbb{Z}_{\ge 0} \cup \{ \infty \}, 
\]
where $\alpha^*(\mathfrak{a})$ denotes the ideal of $K[[t]]$ generated by the image of $\mathfrak{a}$.

\item For $m \in \mathbb{Z}_{\ge 0}$, we define $\operatorname{Cont}^{m}(\mathfrak{a}), \operatorname{Cont}^{\geq m}(\mathfrak{a}) \subset X_{\infty}$ as follows:
\begin{align*}
\operatorname{Cont}^{m}(\mathfrak{a}) &= \{ \alpha \in X_\infty \mid \operatorname{ord}_{\alpha}(\mathfrak a)= m\}, \\
\operatorname{Cont}^{\geq m}(\mathfrak{a}) &= \{ \alpha \in X_\infty \mid \operatorname{ord}_{\alpha}(\mathfrak a)\geq m\}.
\end{align*}
\end{enumerate}
\end{defi}

For cylinders, we can define their codimensions. 
\begin{defi}\label{defi:codim_k}
Let $X$ be a variety over $k$ and let $C \subset X_{\infty}$ be a cylinder. 
\begin{enumerate}
\item 
Assume that $C \subset \operatorname{Cont}^e(\operatorname{Jac}_X)$ for some $e \in \mathbb{Z}_{\ge 0}$.
Then we define the codimension of $C$ in $X_\infty$ as 
\[
\operatorname{codim}(C):=(m+1)\operatorname{dim}X-\operatorname{dim} \left( \psi_m(C) \right)
\]
for any sufficiently large $m$. This definition is well-defined by \cite{EM09}*{Proposition 4.1}.

\item 
In general, we define the codimension of $C$ in $X_\infty$ as follows:
\[
\operatorname{codim}(C):=\min_{e\in\mathbb Z_{\ge 0}} \operatorname{codim} \left( C\cap \operatorname{Cont}^e(\operatorname{Jac}_X) \right). 
\]
By convention, $\operatorname{codim}(C) = \infty$ if $C \cap \operatorname{Cont}^e \left( \operatorname{Jac} _{X} \right) = \emptyset$
for every $e \ge 0$. 
\end{enumerate}
\end{defi}

\begin{rmk}
In what follows, we only consider $k$-arcs (i.e., we assume $K=k$). 
The theory of arc spaces employed in this paper primarily focuses on cylinders and their codimensions. Since the codimension of a cylinder in $X_\infty$ is defined at the level of finite jet schemes and $k$ is an algebraically closed field, 
restricting our attention to $k$-arcs is sufficient for our purposes.
\end{rmk}

\subsection{Arc spaces of $k[t]$-schemes}\label{subsection:arc2}
In this subsection, we deal with the arc spaces of $k[t]$-schemes. 
The reader is referred to \cite{DL02} and \cite{NS22} for details.

Let $X$ be a scheme of finite type  over $k[t]$. 
For a non-negative integer $m$, we define a contravariant functor $F^X_{m}: ({\sf Sch}/k) \to ({\sf Sets})$ by 
\[
F^X _{m}(Y) = \operatorname{Hom} _{k[t]} \left( Y \times_{\operatorname{Spec} k} \operatorname{Spec} k[t]/(t^{m+1}), X \right).
\]
The functor $F^X_{m}$ is always represented by a scheme $X_m$ over $k$. 
We use the same notation $X_{\infty}$, $\psi_m$, and $\pi_{mn}$ in this setting as well: 
\[
X_\infty := \mathop{\varprojlim}\limits_{m} X_m, \qquad \psi_{m}:X_{\infty} \to X_m, \qquad \pi_{mn}:X_m \to X_n. 
\]
We also define a subset $C \subset X_{\infty}$ to be a \textit{cylinder} if it is of the form $C = \psi_{m}^{-1}(S)$ for some $m \ge 0$ and a constructible subset $S \subset X_m$. 
The \textit{contact loci} $\operatorname{Cont}^{m}(\mathfrak{a})$ and $\operatorname{Cont}^{\ge m}(\mathfrak{a})$ are defined in the same way as in Definition \ref{defi:ord_cont}.

\begin{rmk}\label{rmk:bc}
Let $X$ be a $k$-scheme of finite type, and let $X' = X \times _{k} \operatorname{Spec} k[t]$ be its base change. 
Then, we have $X_m \simeq X'_m$ for $m\in \mathbb Z_{\ge 0} \cup \{ \infty \}$. 
Here, $X_m$ denotes the jet and arc schemes as $k$-schemes defined in Subsection \ref{subsection:arc}, and $X'_m$ denotes the jet and arc schemes as $k[t]$-schemes defined in this subsection. Therefore, the theory of arc spaces of $k[t]$-schemes can be seen as a generalization of that for $k$-schemes. 
\end{rmk}

\begin{rmk}\label{rmk:GB}
The schemes $X_m$ and $X_{\infty}$ coincide with the Greenberg schemes $\operatorname{Gr}_m(\mathfrak{X})$ and $\operatorname{Gr}_{\infty}(\mathfrak{X})$ (defined in \cite{Seb04} and \cite{CLNS}) for the formal scheme $\mathfrak{X}$ over $k[[t]]$ associated to the $k[t]$-scheme $X$. See \cite{NS22}*{Remark 2.14} for details.
\end{rmk}

\begin{rmk}
In \cite{NS22}, the authors introduce the condition $(\star)_n$ below, and they study the arc spaces of such $k[t]$-schemes $X$.   
\begin{quote}
$(\star)_n$: \ $X$ is a scheme of finite type over $\operatorname{Spec} k[t]$. 
Any irreducible component of $X$ has dimension at least $n+1$. 
Furthermore, any irreducible component dominating $\operatorname{Spec} k[t]$ is exactly $(n+1)$-dimensional.
\end{quote}
On the other hand, in \cite{NS2}, 
under the motivation of dealing with schemes of finite type over a formal power series ring, 
they relax this condition and treat $X$ with the following conditions:
\begin{quote}
$X$ is a scheme of finite type over $\operatorname{Spec} k[t]$ whose
irreducible components $X_i$ of $X$ have $\dim  X_i \ge n+1$.
\end{quote}
In this subsection, we follow the condition in \cite{NS2} and state the properties for such $k[t]$-schemes, although from Section \ref{section:setting}, we will deal only with schemes satisfying the condition $(\star)_n$.
\end{rmk}

\begin{defi}\label{defi:codim}
Let $n$ be a non-negative integer and 
let $X$ be a scheme of finite type over $k[t]$. 
Suppose that each irreducible component $X_i$ of $X$ has $\dim  X_i \ge n+1$. 
Let $C \subset X_{\infty}$ be a cylinder. 
\begin{enumerate}
\item 
Assume that $C \subset \operatorname{Cont}^e \left( \operatorname{Fitt}^n \left( \Omega _{X/k[t]} \right)\right)$ 
for some $e \in \mathbb{Z}_{\ge 0}$.
Then we define the codimension of $C$ in $X_\infty$ as
\[
\operatorname{codim}(C) := (m+1) n - \operatorname{dim}(\psi_m (C))
\]
for any sufficiently large $m$. This definition is well-defined by \cite{NS2}*{Proposition 5.9(2)} (cf.\ \cite{NS22}*{Lemma 2.15(2)}).

\item 
In general, we define the codimension of $C$ in $X_\infty$ as follows:
\[
\operatorname{codim}(C) := \min_{e \in \mathbb{Z}_{\ge 0}} {\operatorname{codim} \bigl(C \cap 
	\operatorname{Cont}^e \bigl( \operatorname{Fitt}^n \left( \Omega _{X/k[t]} \right) \bigr) \bigr)}. 
\]
By convention, $\operatorname{codim}(C) = \infty$ if 
$C \cap \operatorname{Cont}^e \bigl( \operatorname{Fitt}^n \left( \Omega _{X/k[t]} \right) \bigr) = \emptyset$
for every $e \ge 0$. 
\end{enumerate}
\end{defi}

\begin{rmk}[cf.\ \cite{NS2}*{Remark 5.12}]\label{rmk:codim}
The codimension defined in Definition \ref{defi:codim} depends on the choice of $n$. 
In the remainder of this subsection, we will fix $n$ and use the codimension with respect to the $n$.
\end{rmk}

\begin{defi}\label{defi:thin}
Let $n$ be a non-negative integer, and 
let $X$ be a scheme of finite type over $k[t]$. 
Suppose that each irreducible component $X_i$ of $X$ has $\dim X_i \ge n+1$. 
A subset $A \subset X_{\infty}$ is called \textit{thin} 
if $A \subset Z_{\infty}$ holds for some closed subscheme $Z$ of $X$ with $\dim Z \le n$. 
\end{defi}

\begin{prop}[{cf.\ \cite{Seb04}*{Th\'{e}or\`{e}me 6.3.5}}]\label{prop:negligible}
Let $n$ be a non-negative integer, and 
let $X$ be a scheme of finite type over $k[t]$. 
Suppose that each irreducible component $X_i$ of $X$ has $\dim X_i \ge n+1$. 
Let $C$ be a cylinder in $X_{\infty}$. 
Let $\{ C_{\lambda} \} _{\lambda \in \Lambda}$ be a set of countably many disjoint subcylinders $C_{\lambda} \subset C$. 
If $C \setminus (\bigsqcup _{\lambda \in \Lambda} C_{\lambda}) \subset X_{\infty}$ is a thin set, 
then it follows that 
\[
\operatorname{codim}(C) = \min _{\lambda \in \Lambda} \operatorname{codim}(C_{\lambda}). 
\]
\end{prop}

We define the order of Jacobian for a morphism. 

\begin{defi}\label{defi:jac_k[[t]]}
\begin{enumerate}
\item
Let $X$ and $Y$ be $k[t]$-schemes of finite type, and let $f:X \to Y$ be a morphism over $k[t]$. 
Let $\alpha \in X _{\infty}$ be a $k$-arc and let $\alpha ' := f_{\infty} (\alpha)$. 
Let $S$ be the torsion part of $\alpha ^* \Omega _{X / k[t]}$. 
Then we define the \textit{order} $\operatorname{ord}_{\alpha} (\operatorname{jac}_f)$ 
\textit{of the Jacobian} of $f$ at $\alpha$ as the length of the $k[[t]]$-module
\[
\operatorname{Coker} \bigl( \alpha ^{\prime *} \Omega _{Y / k[t]} \to \alpha ^* \Omega _{X / k[t]} /S \bigr).
\]
In particular, if $\operatorname{ord}_{\alpha} (\operatorname{jac}_f) < \infty$, then we have
\[
\operatorname{Coker} \bigl( \alpha ^{\prime *} \Omega _{Y / k[t]} \to \alpha ^* \Omega _{X / k[t]} /S \bigr) 
\simeq 
\bigoplus _i k[t]/(t^{e_i})
\]
as $k[[t]]$-modules with some positive integers $e_i$ satisfying $\sum_i e_i = \operatorname{ord}_{\alpha} (\operatorname{jac}_f)$.

\item By abuse of notation (see Definition 2.7 and Remark 2.8 in \cite{NS22}), we define 
\[
\operatorname{Cont}^e \left( \operatorname{jac}_f \right) := 
\left \{ \alpha \in X_{\infty} \ \middle | \ \operatorname{ord}_{\alpha} \left( \operatorname{jac}_f \right) = e \right\}
\]
for $e \ge 0$. 
\end{enumerate}
\end{defi}

\begin{lem}[{\cite{NS2}*{Lemma 5.40}, cf.\ \cite{NS22}*{Lemma 2.10}}]\label{lem:additive}
Let $n$ be a non-negative integer, and let $X$, $Y$ and $Z$ be $k[t]$-schemes.
Let $f : X \to Y$ and $g: Y \to Z$ be morphisms over $k[t]$. 
Suppose that each irreducible component $W_i$ of $X$, $Y$ and $Z$ has $\dim W_i \ge n + 1$. 
Let $\alpha \in X_{\infty}$ be an arc, and let $\alpha ' := f_{\infty}(\alpha)$. 
Suppose that 
\[
\operatorname{ord}_{\alpha} \left( \operatorname{Fitt}^n \left( \Omega _{X/k[t]} \right) \right) < \infty, \quad 
\operatorname{ord}_{\alpha '} \left( \operatorname{Fitt}^n \left( \Omega _{Y/k[t]} \right) \right) < \infty. 
\]
Then we have 
\[
\operatorname{ord}_{\alpha} \left( \operatorname{jac}_{g \circ f} \right) = 
\operatorname{ord}_{\alpha} \left( \operatorname{jac}_{f} \right) + 
\operatorname{ord}_{\alpha '} \left( \operatorname{jac}_{g} \right). 
\]
\end{lem}

\begin{prop}[{\cite{DL02}*{Lemma 1.17, Remark 1.19}, \cite{NS2}*{Proposition 5.43}, \cite{NS22}*{Proposition 2.33}}]\label{prop:EM6.2_k[t]}
Let $n$ be a non-negative integer. 
Let $X$ and $Y$ be $k[t]$-schemes.
Let $f: X \to Y$ be a morphism over $k[t]$. 
Suppose that each irreducible component $W_i$ of $X$ and $Y$ has $\dim W_i \ge n + 1$. 
Let $e,e',e''\in \mathbb Z_{\ge 0}$. 
Let $A \subset X_{\infty}$ be a cylinder and let $B = f_{\infty}(A)$. 
Assume that 
\[
A  \subset \operatorname{Cont}^{e''} \left( \operatorname{Fitt}^n \left( \Omega _{X/k[t]} \right)  \right) 
	\cap \operatorname{Cont}^{e} \left( \operatorname{jac}_f \right), \quad
B  \subset \operatorname{Cont}^{e'} \left( \operatorname{Fitt}^n \left( \Omega _{Y/k[t]} \right)  \right). 
\]
Then, $B$ is a cylinder of $Y_{\infty}$. 
Moreover, if $f_{\infty}|_{A}$ is injective, then it follows that
\[
\operatorname{codim}(A) + e = \operatorname{codim}(B).
\]
\end{prop}

\begin{prop}[{\cite{DL02}*{Lemma 3.5}, \cite{NS2}*{Proposition 5.44}, \cite{NS22}*{Proposition 2.35}}]\label{prop:DL2_k[t]}
Let $n$ be a non-negative integer. Let $X$ be a $k[t]$-scheme. 
Suppose that a finite group $G$ acts on a $k[t]$-scheme $X$. 
Suppose that each irreducible component $X_i$ of $X$ has $\dim X_i \ge n + 1$. 
Let $f: X \to Y := X/G$ be the quotient morphism. 
Let $A \subset X_{\infty}$ be a $G$-invariant cylinder and let $B = f_{\infty} (A)$. 
Let $e,e',e''\in \mathbb Z_{\ge 0}$. 
Assume that 
\[
A  \subset \operatorname{Cont}^{e''} \left( \operatorname{Fitt}^n \left( \Omega _{X/k[t]} \right)  \right) 
	\cap \operatorname{Cont}^{e} \left( \operatorname{jac}_f \right), \quad
B  \subset \operatorname{Cont}^{e'} \left( \operatorname{Fitt}^n \left( \Omega _{Y/k[t]} \right)  \right). 
\]
Then $B$ is a cylinder of $Y_{\infty}$ with 
\[
\operatorname{codim}(A)+e=\operatorname{codim}(B).
\]
\end{prop}

\begin{lem}[{\cite{NS22}*{Lemma 2.34}}]\label{lem:EM8.4}
Let $n$ and $N$ be non-negative integers with $n \le N$. 
Set $c := N - n$. 
Let $f_1, \ldots , f_c \in k[x_1,\ldots,x_N][t]$, and 
let $I_X = ( f_1, \ldots , f_c ) \subset k[x_1,\ldots,x_N][t]$ be the ideal generated by them. 
We set 
\[
A := \operatorname{Spec} k[x_1,\ldots,x_N][t], \quad 
X := \operatorname{Spec} \left( k[x_1,\ldots,x_N][t] / I_X \right). 
\]
Let $C \subset A_{\infty}$ be an irreducible locally closed cylinder. 
If 
\begin{itemize}
\item $C \subset \operatorname{Cont}^{\ge d} \bigl( f_1 \cdots f_c \bigr)$ and 
\item $C \cap X_{\infty} \cap \operatorname{Cont}^{e}\left( \operatorname{Fitt}^n \left( \Omega _{X/k[t]} \right) \right) \not = \emptyset$
\end{itemize}
hold for some $d \ge 0$ and $e \ge 0$, 
then it follows that
\[
\operatorname{codim}_{X_{\infty}} (C \cap X_{\infty}) \le \operatorname{codim}_{A_{\infty}} (C) + e - d.
\]
\end{lem}
\begin{proof}
We set $I :=  \left\{ 
(d_1, \ldots , d_c) \in \mathbb{Z}_{\ge 0} ^c \ \middle | \ d_1 + \cdots + d_c = d  
\right \}$. 
For ${\bf d} \in I$, we define $C _{\bf d} \subset X_{\infty}$ by 
\[
C _{\bf d} := C \cap \bigcap_{i=1}^c \operatorname{Cont}^{\ge d_i} (f_i). 
\]
Since $C = \bigcup _{{\bf d} \in I} C_{\bf d}$, the assertion follows from the inequality
\[
\operatorname{codim}_{X_{\infty}} \left( C _{\bf d} \cap X_{\infty} \right) 
\le 
\operatorname{codim}_{A_{\infty}} \left( C_{\bf d} \right) + e - d,
\]
which is proved by \cite{NS22}*{Lemma 2.34}. 
\end{proof}

\section{Settings and Denef and Loeser’s theory}\label{section:setting}
Let $N$ be a positive integer. Set $R := k[x_1, \ldots , x_N]$. 
Let $d$ be a positive integer and let $\xi \in k$ be a primitive $d$-th root of unity. 
Let $G \subset \operatorname{GL}_N(k)$ be a finite subgroup with order $d$ that linearly acts on 
$\overline{A} := \mathbb{A}^N _k = \operatorname{Spec} R$. 
We denote by 
\[
A := \overline{A}/G = \operatorname{Spec} R^G
\]
the quotient scheme. 
Let $Z \subset A$ be the minimum closed subset such that $\overline{A} \to A$ is \'{e}tale outside $Z$.
We assume that $\operatorname{codim}_A Z \ge 2$, and hence the quotient map $\overline{A} \to A$ is \'{e}tale in codimension one. 

\begin{defi}\label{defi:semiinv}
Let $f\in R$.
\begin{enumerate}
\item
We say that $f$ is \textit{$G$-semi-invariant} if for any $\gamma \in G$, there exists $0 \le a \le d-1$ such that $\gamma (f) =\xi^af$.

\item
Suppose that $f$ is $G$-semi-invariant. 
For $\gamma \in G$, we define $w_{\gamma}(f):=\frac{a}{d}$ when $\gamma (f) =\xi^a f$ with $0 \le a \le d-1$.
\end{enumerate}
\end{defi}

Let $c$ be a non-negative integer with $c \le N$. 
We set $n := N - c$. 
Let $f_1,\ldots,f_c \in R$ be a regular sequence which is contained in the maximal ideal $(x_1, \ldots , x_N)$ at the origin. 
Suppose that $f_1, \ldots, f_c$ are $G$-semi-invariant.
We set 
\[
I_{\overline{B}} := \left( f_1, \ldots , f_c \right), \quad 
\overline{B} := \operatorname{Spec} \left( R/I_{\overline{B}} \right). 
\]
We denote 
\[
I_B := I_{\overline{B}} \cap R^G, \quad 
B := \overline{B}/G = \operatorname{Spec} \left( R^G/I_B \right). 
\]
We assume that $B$ is normal and $\operatorname{codim} _B (B \cap Z) \ge 2$. 
\begin{lem}\label{lem:etaleness of B}
\begin{enumerate}
\item
The quotient map $\overline{B} \to B$ is \'{e}tale outside $Z$.
\item
$\overline{B}$ is normal. 
\end{enumerate}
\end{lem}
\begin{proof}
We prove (1). Let $q:\overline{A} \to A$ denote the quotient map. 
Let $q^{-1}(B) \subset \overline{A}$ be the scheme theoretic inverse image. 
Then we have $\overline{B} \subset q^{-1}(B)$. 
Since $B$ is reduced and $q^{-1}(B) \to B$ is \'{e}tale outside $Z$, 
it follows that $q^{-1}(B) \setminus q^{-1}(Z)$ is reduced. 
Since $\overline{B} _{\rm red} = (q^{-1}(B)) _{\rm red}$, we have $q^{-1}(B) \setminus q^{-1}(Z) = \overline{B} \setminus q^{-1}(Z)$. Therefore, we conclude that $\overline{B} \to B$ is \'{e}tale outside $Z$. 

We prove (2). Since $B$ is normal, we have $\operatorname{codim} _B (B_{\rm sing}) \ge 2$. 
Therefore, by (1) and the assumption $\operatorname{codim} _B (B \cap Z) \ge 2$, we have $\operatorname{codim} _{\overline{B}} \left( \overline{B}_{\rm sing} \right) \ge 2$. 
Since the sequence $f_1, \ldots, f_c$ is a regular sequence, the normality of $\overline{B}$ follows from Serre's criterion. 
\end{proof}

\begin{rmk}\label{rmk:codim 2}
In \cite{NS22}, the condition $\operatorname{codim} _B (B \cap Z) \ge 2$ is not assumed. 
In fact, if $f_1, \ldots, f_c$ are $G$-invariant, 
then the condition $\operatorname{codim} _B (B \cap Z) \ge 2$ follows from
the normality of $B$ and the condition $\operatorname{codim} _A  Z \ge 2$ (see \cite{NS22}*{Lemma 4.2}). 
However, the following example shows that this is false in general 
when $f_1, \ldots, f_c$ are just assumed to be $G$-semi-invariant.

For $\overline{A}=\operatorname{Spec} \bigl( k[x,y] \bigr)$, 
$G=\langle\operatorname{diag}(-1,-1)\rangle$ and 
$\overline{B}=\operatorname{Spec} \bigl( k[x,y]/(x) \bigr)=\operatorname{Spec} \bigl( k[y] \bigr)$, we have
\begin{align*}
A&=\overline{A}/G=\operatorname{Spec} \bigl( k[x^2,xy,y^2] \bigr),\\
B&=\overline{B}/G=\operatorname{Spec} \bigl( k[x^2,xy,y^2]/(x^2, xy) \bigr)=\operatorname{Spec} \bigl( k[y^2] \bigr), \\
Z&=\operatorname{Spec} \bigl( k[x^2,xy,y^2] /(x^2,xy,y^2)\bigr).
\end{align*}
Then, $\overline{A} \to A$ is \'{e}tale outside $Z$ and $\operatorname{codim}_A Z = 2$.
However, we have $\operatorname{codim} _B (B \cap Z) =1$.
Moreover, $\overline{B} \to B$ is not \'{e}tale in codimension one.
\end{rmk}
We take a positive integer $r$ such that $\omega _{B}^{[r]}$ is invertible.

Let $\gamma \in G$. 
Since $G$ is a finite group, $\gamma$ can be diagonalized with respect to a new basis $x_1^{(\gamma)}, \ldots, x_N^{(\gamma)}$ consisting of $k$-linear combinations of the original basis $x_1, \ldots , x_N$.
Let $\operatorname{diag} \left( \xi ^{e_1}, \ldots , \xi ^{e_N} \right)$ be the diagonal matrix 
with $0 \le e_i \le d-1$. 
Then, the \textit{age} $\operatorname{age}(\gamma)$ of $\gamma$ is defined by $\operatorname{age}(\gamma) := \frac{1}{d} \sum _{i = 1} ^N e_i$. 
We also define a $k$-algebra homomorphism ${\lambda}^* _{\gamma}:R \to R[t^{1/d}]$ as follows:
\[
{\lambda}^* _{\gamma}: R \to R[t^{1/d}]; 
\quad x_i^{(\gamma)} \mapsto t^{e_i/d}x_i^{(\gamma)}. 
\]
We note that the ring homomorphism ${\lambda}^* _{\gamma}$ does not depend on the choice of the basis $x_1^{(\gamma)}, \ldots, x_N^{(\gamma)}$ as long as $\gamma$ is diagonalized with respect to this basis. 
Let $C_{\gamma}$ denote the centralizer of $\gamma$ in $G$. 

\begin{rmk}
In \cite{NS22}*{Section 3}, $\lambda^*_{\gamma}$ is defined as a $k[t]$-ring homomorphism from $R[t]^G$ to $R[t]^{C_{\gamma}}$ satisfying $\lambda^*_{\gamma} (x_i^{(\gamma)}) = t^{e_i/d}x_i^{(\gamma)}$. 
Although the domain and codomain are different, we use the same symbol because the map in \cite{NS22} is the restriction of the base change of our $\lambda^*_{\gamma}$ to $R[t]^G$.
\end{rmk}

\begin{lem}\label{lem:lambda}
The following assertions hold. 
\begin{enumerate}
\item 
For any $\delta \in C_{\gamma}$, the composition $R \xrightarrow{\delta} R \xrightarrow{\lambda ^* _{\gamma}} R[t^{1/d}]$ coincides with the composition $R \xrightarrow{\lambda ^* _{\gamma}} R[t^{1/d}] \xrightarrow{\delta} R[t^{1/d}]$. 
Here, $R[t^{1/d}] \xrightarrow{\delta} R[t^{1/d}]$ is the $k[t^{1/d}]$-ring homomorphism induced by $R \xrightarrow{\delta} R$.

\item
We have 
$\lambda^* _{\gamma} \left( R^G \right) \subset R^{C_{\gamma}}[t]$. 

\item 
If $f \in R$ is a $G$-semi-invariant element, then 
for any $\delta \in G$, we have $\lambda ^* _{\delta}(f)t^{-w_{\delta}(f)} \in R[t]$.
\end{enumerate}
\end{lem}
\begin{proof} 
Let $\delta \in C_{\gamma}$. Then, by replacing the bases $x_1^{(\gamma)}, \ldots, x_N^{(\gamma)}$, we may assume that $\delta$ is also a diagonal matrix with respect to this basis. 
Then, the assertion (1) is obvious. 

(2) follows from (1) (cf.\ \cite{DL02}*{Lemma 2.6} and \cite{NS22}*{Proposition 3.7}). (3) easily follows from the definitions. 
\end{proof}

By Lemma \ref{lem:lambda}(2), $\lambda^* _{\gamma}$ induces a ring homomorphism 
$R^G \to R[t]$. By abuse of notation, we also use the same symbol $\lambda^* _{\gamma}$ for this ring homomorphism. 
We define 
\begin{align*}
I_{\overline{B}^{(\gamma)}} 
&:= \left( \lambda ^* _{\gamma}(f_1)t^{-w_\gamma(f_1)}, \ldots , \lambda ^* _{\gamma}(f_c)t^{-w_\gamma(f_c)} \right)
\subset R[t], \\ 
\overline{B}^{(\gamma)} 
&:= \operatorname{Spec} \left(R[t]/I_{\overline{B}^{(\gamma)}}\right). 
\end{align*}

\begin{lem}\label{lem:incl}
For the ring homomorphism $\lambda^* _{\gamma} : R^G \to R[t]$, the following assertions hold. 
\begin{enumerate}
\item 
$\lambda ^* _{\gamma}(I_B) \subset I_{\overline{B}^{(\gamma)}}$.

\item 
$t I_{\overline{B}^{(\gamma)}} \subset \sqrt{\lambda ^* _{\gamma}(I_B) R[t]}$. 

\item
$I_{\overline{B}^{(\gamma)}}$ is $C_{\gamma}$-invariant.
\end{enumerate}
\end{lem}
\begin{proof}
Consider the following commutative diagram of rings and their ideals:
\[
\xymatrix{
R^G \ar@{^(->}[r] \ar@{}[d]|{\bigcup} & R \ar@{}[d]|{\bigcup} \ar[rr]^-{x_i^{(\gamma)} \mapsto t^{e_i/d} x_i^{(\gamma)}}&& R[t^{1/d}] \ar@{}[d]|{\bigcup} \\ 
I_B \ar[r] & I_{\overline{B}} \ar[drr] && I_{\overline{B}^{(\gamma)}} R[t^{1/d}] \ar@{}[d]|{\bigcup} \\
&&& \left( \lambda ^* _{\gamma}(f_1), \ldots , \lambda ^* _{\gamma}(f_c) \right)
}
\]
Since the composition $R^G \xrightarrow{\lambda ^* _{\gamma}} R[t] \hookrightarrow R[t^{1/d}]$ is equal to the composition of the two ring homomorphisms in the diagram above, we have
\[
\lambda ^* _{\gamma}(I_B) R[t^{1/d}] \subset I_{\overline{B}^{(\gamma)}} R[t^{1/d}]. 
\]
By taking their restrictions to $R[t]$, we have 
\[
\lambda ^* _{\gamma}(I_B) R[t]
= \left( \lambda ^* _{\gamma}(I_B) R[t^{1/d}] \right) \cap R[t] 
\subset \left(I_{\overline{B}^{(\gamma)}} R[t^{1/d}] \right) \cap R[t]
= I_{\overline{B}^{(\gamma)}}, 
\]
which proves (1). 

Note that we have $f_i^d \in R^G$ for each $1 \le i \le c$ since $f_i$ is $G$-semi-invariant. 
Hence, we have $f_i ^d \in I_B$. 
Since $d \ge d w _{\gamma} (f_i)$ and $f_i ^d \in I _B$, we have 
\[
\bigl( t \cdot \lambda ^* _{\gamma}(f_i)t^{-w_\gamma(f_i)} \bigr) ^d
= t^{d - d w _{\gamma} (f_i)} \lambda ^* _{\gamma}(f_i ^d)
\in \lambda ^* _{\gamma}(I_B) R[t], 
\]
which proves (2). 

(3) follows from Lemma \ref{lem:lambda}(1). 
\end{proof}

By Lemmas \ref{lem:lambda}(2) and \ref{lem:incl}(1), we have the following diagram of rings: 
\[
  \xymatrix{
R^{G}[t] \ar[r] ^-{x_i^{(\gamma)} \mapsto t^{e_i/d} x_i^{(\gamma)}} \ar@{->>}[d] & R^{C_{\gamma}}[t] \ar@{->>}[d] \ar@{^{(}->}[r] & R[t] \ar@{->>}[d] \\
R^G[t]/ \left( I_B R^G[t] \right)  \ar[r] & R^{C_{\gamma}}[t]/ \left( I_{\overline{B} ^{(\gamma)}} \cap R^{C_{\gamma}}[t] \right) \ar[r]   & R[t] / I_{\overline{B} ^{(\gamma)}}
}
\]
Then, we have the following commutative diagram of $k[t]$-schemes. 
\[
  \xymatrix{
A \times \mathbb{A}^1 _k & \widetilde{A}^{(\gamma)} \ar[l]^-{\lambda _{\gamma}} & \overline{A}^{(\gamma)} \ar[l] \ar@/_20pt/[ll]_-{\overline{\lambda} _{\gamma}} \\
B \times \mathbb{A}^1 _k \ar@{^{(}->}[u] & \widetilde{B}^{(\gamma)} \ar@{^{(}->}[u] \ar[l]_-{\mu _{\gamma}} &  \overline{B}^{(\gamma)} \ar[l]_-{p} \ar@{^{(}->}[u] \ar@/^20pt/[ll]^-{\overline{\mu} _{\gamma}}
  }
\]
Here, we defined $k[t]$-schemes $\widetilde{A}^{(\gamma)}$, $\overline{A}^{(\gamma)}$, $\widetilde{B}^{(\gamma)}$ and $\overline{B}^{(\gamma)}$ as follows: 
\begin{alignat*}{2}
\widetilde{A}^{(\gamma)} &:= \operatorname{Spec} \left( R^{C_{\gamma}}[t] \right), &
	\quad \overline{A}^{(\gamma)} &:= \operatorname{Spec} (R[t]) \\
\widetilde{B}^{(\gamma)} &:= \operatorname{Spec} \left( R^{C_{\gamma}}[t]/ \bigl ( I_{\overline{B} ^{(\gamma)}} \cap R^{C_{\gamma}}[t] \bigr ) \right), \quad &
	\overline{B}^{(\gamma)} &:= \operatorname{Spec} \bigl( R[t] / I_{\overline{B} ^{(\gamma)}} \bigr). 
\end{alignat*}
Let $\widetilde{A}^{(\gamma)}_{\infty}$, $\overline{A}^{(\gamma)} _{\infty}$, $\widetilde{B}_{\infty}^{(\gamma)}$ and $\overline{B}_{\infty}^{(\gamma)}$ be their arc spaces as $k[t]$-schemes. 
Let $A_{\infty}$ and $B_{\infty}$ be the arc spaces of $A$ and $B$ as $k$-schemes. 
By Remark \ref{rmk:bc}, $A_{\infty}$ and $B_{\infty}$ can be identified with the arc spaces of $A \times \mathbb{A}^1 _k$ and $B \times \mathbb{A}^1 _k$ as $k[t]$-schemes. 
Then we have the following diagram of arc spaces: 
\[
\xymatrix{
A_{\infty} & \widetilde{A}^{(\gamma)}_{\infty} \ar[l]^-{\lambda _{\gamma \infty}} & \overline{A}^{(\gamma)} _{\infty} \ar[l] \ar@/_20pt/[ll]_{\overline{\lambda} _{\gamma \infty}} \\
B_{\infty} \ar@{^{(}-{>}}[u] & \widetilde{B}_{\infty} ^{(\gamma)} \ar[l]_-{\mu _{\gamma \infty}} \ar@{^{(}-{>}}[u] & \overline{B}_{\infty} ^{(\gamma)} \ar[l]_-{p _{\infty}} \ar@{^{(}-{>}}[u] \ar@/^20pt/[ll]^-{\overline{\mu} _{\gamma \infty}}
}
\]
We identify $B_{\infty}$, $\widetilde{B}_{\infty} ^{(\gamma)}$ and $\overline{B}_{\infty} ^{(\gamma)}$ with the corresponding closed subset of $A_{\infty}$, $\widetilde{A}^{(\gamma)}_{\infty}$ and $\overline{A}^{(\gamma)} _{\infty}$ (cf.\ \cite{NS22}*{Lemma 2.28(1)}). 

\begin{rmk}\label{rmk:seisitu}
\begin{enumerate}
\item 
By Lemma \ref{lem:incl}(3), $\overline{B}^{(\gamma)}$ is a $C_{\gamma}$-invariant subscheme of $\overline{A}^{(\gamma)}$. 
Furthermore, we have $\overline{B}^{(\gamma)} / C_{\gamma} = \widetilde{B}^{(\gamma)}$. 

\item 
Since $I _{\overline{B}^{(\gamma)}}$ is generated by $c$ elements, 
each irreducible component $W_i$ of $\overline{B}^{(\gamma)}$ satisfies $\dim W_i \ge n+1$. 
By (1), the same property also holds for $\widetilde{B}^{(\gamma)}$. 
Therefore, we can apply lemmas and propositions in Subsection \ref{subsection:arc2}. 
In Sections \ref{section:mld} and \ref{section:PIA}, we will consider the codimensions of cylinders in $\overline{B}^{(\gamma)} _{\infty}$ with respect to $n$, and cylinders in $\overline{A}^{(\gamma)} _{\infty}$ with respect to $N$, respectively (cf.\ Remark \ref{rmk:codim}). 

\item 
By the same argument as in \cite{NS22}*{Remark 4.3}, we have a surjective \'{e}tale morphism
\[
\overline{B} \times \left( \mathbb{A}^1 _k \setminus \{ 0 \} \right)
\to \overline{B}^{(\gamma)}_{t \not = 0}
\]
(see Remark \ref{rmk:seisitu2} for more detail). 
Therefore, $\overline{B}^{(\gamma)}_{t \not = 0}$ and $\widetilde{B}^{(\gamma)}_{t \not = 0}$ are integral $k[t]$-schemes of dimension $n+1$. 
In particular, $\overline{B}^{(\gamma)}$ and $\widetilde{B}^{(\gamma)}$ each have only one irreducible component which dominates $\operatorname{Spec} k[t]$. 
By \cite{NS22}*{Remark 2.5}, if $W$ is the dominant component of $\overline{B}^{(\gamma)}$ (resp.\ $\widetilde{B}^{(\gamma)}$), we have $\overline{B}^{(\gamma)} _{\infty} = W_{\infty}$ (resp.\ $\widetilde{B}^{(\gamma)} _{\infty} = W_{\infty}$). 
\end{enumerate}
\end{rmk}

\begin{lem}\label{lem:pull back}
We have
\[
\lambda _{\gamma \infty} ^{-1} (B_{\infty}) = \widetilde{B}_{\infty} ^{(\gamma)}, \qquad 
\overline{\lambda} _{\gamma \infty} ^{-1} (B_{\infty}) = \overline{B}_{\infty} ^{(\gamma)}. 
\] 
\end{lem}
\begin{proof}
Suppose that $\alpha \in \overline{A}^{(\gamma)}_{\infty}$ satisfies $\overline{\lambda} _{\gamma \infty}(\alpha)\in B_{\infty}$. 
Let $\alpha ^*: R[t] \to k[[t]]$ be the $k[t]$-ring homomorphism corresponding to $\alpha$. 
Since $\overline{\lambda} _{\gamma \infty}(\alpha)\in B_{\infty}$, we have 
\[
\alpha ^* \left( \lambda^* _{\gamma}(I_B) R[t] \right) =0.
\]
By Lemma \ref{lem:incl}(2), we obtain
\[
\alpha ^* \left( tI_{\overline{B} ^{(\gamma)}} \right)
\subset \alpha ^* \left( \sqrt{\lambda^* _{\gamma}(I_B) R[t]} \right)=0.
\]
Since $\alpha ^*$ is a $k[t]$-ring homomorphism, 
we conclude $\alpha ^* \left( I_{\overline{B} ^{(\gamma)}} \right)=0$, which proves 
$\alpha \in \overline{B}_{\infty} ^{(\gamma)}$.
We have proved that $\overline{\lambda} _{\gamma \infty} ^{-1} (B_{\infty}) = \overline{B}_{\infty} ^{(\gamma)}$.

By Lemma \ref{lem:incl}(2), we have 
\begin{align*}
t \left( I_{\overline{B} ^{(\gamma)}} \cap R^{C_{\gamma}}[t] \right)
&\subset t I_{\overline{B}^{(\gamma)}} \cap R^{C_{\gamma}}[t] \\
&\subset \sqrt{ \lambda ^* _{\gamma}(I_B) R[t]} \cap R^{C_{\gamma}}[t]\\
&=\sqrt{\lambda ^* _{\gamma}(I_B) R^{C_{\gamma}}[t]}, 
\end{align*}
which proves $\lambda _{\gamma \infty} ^{-1} (B_{\infty}) = \widetilde{B}_{\infty} ^{(\gamma)}$ by the same argument above. 
\end{proof}

The following proposition is due to Denef and Loeser \cite{DL02}, 
and this allows us to study the arc space $A_{\infty}$ via the arc space $\overline{A}^{(\gamma)}_{\infty}$. 

\begin{prop}[{\cite{DL02}*{Section 2}, cf.\ \cite{NS22}*{Subsections 3.1, 3.2}}]\label{prop:DL_A}
The ring homomorphism $\lambda^* _{\gamma}$ induces the maps 
$\lambda_{\gamma \infty}: \widetilde{A}^{(\gamma)}_{\infty} \to A_{\infty}$ and 
$\overline{\lambda}_{\gamma \infty}: \overline{A}^{(\gamma)} _{\infty} \to A_{\infty}$, 
and the following hold. 
\begin{enumerate}
\item There is a natural inclusion $\overline{A}^{(\gamma)} _{\infty}/C_{\gamma} \hookrightarrow \left( \overline{A}^{(\gamma)} /C_{\gamma} \right)_{\infty} = \widetilde{A}^{(\gamma)}_{\infty}$. 

\item The composite map $\overline{A}^{(\gamma)} _{\infty}/C_{\gamma} \hookrightarrow \widetilde{A}^{(\gamma)}_{\infty}\xrightarrow{\lambda _{\gamma \infty}} A_{\infty}$ is injective outside $Z_{\infty}$. 

\item $\bigsqcup _{\langle \gamma \rangle \in \operatorname{Conj}(G)} 
\left( \overline{\lambda}_{\gamma \infty} \bigl( \overline{A}^{(\gamma)} _{\infty} \bigr) \setminus Z_{\infty} \right)
= A_{\infty} \setminus Z_{\infty}$ holds, where $\operatorname{Conj}(G)$ denotes the set of the conjugacy classes of $G$. 
\end{enumerate}
\end{prop}

\noindent
By Lemma \ref{lem:pull back}, we can deduce the same statement for $B$. 
\begin{prop}[{cf.\ \cite{NS22}*{Subsection 3.3}}]\label{prop:DL_X}
The ring homomorphism $\lambda^* _{\gamma}$ induces the maps 
$\mu_{\gamma \infty}:\widetilde{B}_{\infty} ^{(\gamma)} \to B_{\infty}$ and 
$\overline{\mu}_{\gamma \infty}: \overline{B}^{(\gamma)}_{\infty} \to B_{\infty}$, 
and the following hold. 
\begin{enumerate}
\item There is a natural inclusion $\overline{B}_{\infty} ^{(\gamma)} / C_{\gamma} \hookrightarrow \widetilde{B}_{\infty} ^{(\gamma)}$. 

\item The composite map $\overline{B}_{\infty} ^{(\gamma)} / C_{\gamma} \hookrightarrow \widetilde{B}_{\infty} ^{(\gamma)} 
\xrightarrow{\mu _{\gamma \infty}} B_{\infty}$ is injective outside $Z_{\infty}$. 

\item $\bigsqcup _{\langle \gamma \rangle \in \operatorname{Conj}(G)} 
\left( \overline{\mu}_{\gamma \infty} \bigl( \overline{B}^{(\gamma)}_{\infty} \bigr) \setminus Z_{\infty} \right)
= B_{\infty} \setminus Z_{\infty}$ holds. 
\end{enumerate}
\end{prop}

\section{Order calculations}\label{section:order}
We keep the notation from Section \ref{section:setting}. 
We define 
\[
\operatorname{Jac}_{\overline{B}^{(\gamma)} /k[t]} := 
\operatorname{Fitt}^{n} \left( \Omega _{\overline{B}^{(\gamma)} /k[t]} \right), \quad 
\operatorname{Jac}_{\widetilde{B}^{(\gamma)} /k[t]} := 
\operatorname{Fitt}^{n} \left( \Omega _{\widetilde{B}^{(\gamma)} /k[t]} \right). 
\]
The purpose of this section is to prove Lemmas \ref{lem:L3} and \ref{lem:L5}, 
which correspond to Lemmas 4.6 and 4.7 in \cite{NS22}. 
There, relations among the orders of 
$\operatorname{jac}_{\overline{\mu} _{\gamma}}$, 
$\operatorname{jac}_{p}$, 
$\operatorname{Jac}_{\overline{B}^{(\gamma)}/k[t]}$ and
$\mathfrak{n}_{r,B} \mathcal{O}_{\overline{B}^{(\gamma)}}$ are revealed. 
For this purpose, we will extend the schemes $\overline{B}^{(\gamma)}$ and $\widetilde{B}^{(\gamma)}$ over $k[t^{1/d}]$, 
and we will work on them.

We have the following commutative diagram of rings and ideals: 
\[
\xymatrix{
R^G[t^{1/d}] \ar@{^(->}[r] & R[t^{1/d}] \ar[r]^{x_i^{(\gamma)} \mapsto t^{e_i/d} x_i^{(\gamma)}} & R[t^{1/d}] \\
I_B R^G[t^{1/d}] \ar[r] \ar@{}[u]|{\bigcup} & I_{\overline{B}}R[t^{1/d}] \ar@{}[u]|{\bigcup} \ar[dr] &  I_{\overline{B}^{(\gamma)}} R[t^{1/d}] \ar@{}[u]|{\bigcup} \\
& & \left( \lambda ^* _{\gamma}(f_1), \ldots , \lambda ^* _{\gamma}(f_c) \right) \ar@{}[u]|{\bigcup}
}
\]
By abuse of notation, we denote the ideal $\left( \lambda ^* _{\gamma}(f_1), \ldots , \lambda ^* _{\gamma}(f_c) \right)$ in the diagram by $\lambda ^*_{\gamma} \left( I_{\overline{B}} \right) R[t^{1/d}]$. 
Then, we have the following diagram of rings. 
\[
\xymatrix{
R^G[t^{1/d}] \ar@{^(->}[r] \ar@{->>}[d] & R[t^{1/d}] \ar[r]^{x_i^{(\gamma)} \mapsto t^{e_i/d} x_i^{(\gamma)}} \ar@{->>}[d] & R[t^{1/d}] \ar@{->>}[d] \\
R^G[t^{1/d}]/\left( I_B R^G[t^{1/d}] \right) \ar[r] & R[t^{1/d}]/\left(  I_{\overline{B}}R[t^{1/d}] \right) \ar[r] & R[t^{1/d}]/\left(  \lambda ^*_{\gamma} \left( I_{\overline{B}} \right) R[t^{1/d}]\right) \ar@{->>}[d]\\
&& R[t^{1/d}] \big/ \left(  I_{\overline{B}^{(\gamma)}} R[t^{1/d}] \right)  \\
}
\]
Then, we have the corresponding diagram of $k[t^{1/d}]$-schemes:
\[
  \xymatrix{
A_{k[t^{1/d}]} = A \times \mathbb{A}^1_k & \overline{A}_{k[t^{1/d}]} =  \overline{A} \times \mathbb{A}^1_k \ar[l] & \overline{A}^{(\gamma)}_{1/d} = \mathbb{A}^N _k \times \mathbb{A}^1_k \ar[l] \\
B_{k[t^{1/d}]} = B \times \mathbb{A}^1_k  \ar@{^(->}[u] & \overline{B}_{k[t^{1/d}]} = \overline{B} \times \mathbb{A}^1_k \ar[l] \ar@{^(->}[u] & \widehat{B}^{(\gamma)}_{1/d} \ar[l] \ar@{^(->}[u] \\
&& \overline{B}^{(\gamma)}_{1/d} \ar@{^(->}[u]
  }
\]
Here, we considered the following $k[t]$-schemes:
\begin{alignat*}{2}
A_{k[t^{1/d}]} &:= \operatorname{Spec} \bigl ( R^G[t^{1/d}] \bigr),& \quad
	&\overline{A}_{k[t^{1/d}]} := \overline{A}^{(\gamma)}_{1/d} := \operatorname{Spec} \bigl ( R[t^{1/d}] \bigr), \\
B_{k[t^{1/d}]} &:= \operatorname{Spec} \left( R^G[t^{1/d}] \big/ \bigl ( I_B R^G[t^{1/d}] \bigr) \right),& \quad 
	& \overline{B}_{k[t^{1/d}]}:= \operatorname{Spec} \left( R[t^{1/d}] \big/\bigl(  I_{\overline{B}}R[t^{1/d}] \bigr) \right), \\
\widehat{B}^{(\gamma)}_{1/d} &:= \operatorname{Spec} \left( R[t^{1/d}] \big/ \bigl(  \lambda ^*_{\gamma} \left( I_{\overline{B}} \right) R[t^{1/d}] \bigr) \right),& \quad
& \overline{B}^{(\gamma)}_{1/d}:= \operatorname{Spec} \left( R[t^{1/d}] \big/\bigl(  I_{\overline{B}^{(\gamma)}} R[t^{1/d}] \bigr) \right). 
\end{alignat*}
Then, we have 
\[
\overline{B}^{(\gamma)}_{1/d} = \overline{B}^{(\gamma)} \times _{\operatorname{Spec}k[t]} \operatorname{Spec}k[t^{1/d}]. 
\]
\begin{rmk}\label{rmk:seisitu2}
By definition, the morphisms $\overline{B}^{(\gamma)}_{1/d} \to \widehat{B}^{(\gamma)}_{1/d}$ and $\widehat{B}^{(\gamma)}_{1/d} \to \overline{B}_{k[t^{1/d}]}$ are isomorphic outside $t^{1/d} = 0$. Then, the \'{e}tale morphism $\overline{B} \times \left( \mathbb{A}^1 _k \setminus \{ 0 \} \right) \to \overline{B}^{(\gamma)}_{t \not = 0}$ in Remark \ref{rmk:seisitu}(3) is given by the compositions of the following mophisms: 
\[
  \xymatrix{
  \bigl( \overline{B}_{k[t^{1/d}]} \bigr)_{t^{1/d} \not = 0}  & \bigl( \overline{B}^{(\gamma)}_{1/d} \bigr)_{t^{1/d} \not = 0} \ar[l]_{\simeq} \\
  \overline{B} \times \operatorname{Spec} \left( k[t^{1/d}, t^{-1/d}] \right) \ar@{=}[u] & \overline{B}^{(\gamma)} \times _{\operatorname{Spec}k[t]} \operatorname{Spec}k[t^{1/d}, t^{-1/d}] \ar@{=}[u] \ar[d]^{\text{\'{e}tale}} \\
  \overline{B} \times \left( \mathbb{A}^1 _k \setminus \{ 0 \} \right) \ar@{=}[u] & \overline{B}^{(\gamma)} \times _{\operatorname{Spec}k[t]} \operatorname{Spec}k[t,t^{-1}]  \\
  & \overline{B}^{(\gamma)}_{t \not = 0} \ar@{=}[u]
  }
\]
Furthermore, for any $a, b \in k^{\times}$ with $a = b^d$, it induces an isomorphism 
\[\tag{i}
\overline{B} \simeq \overline{B} \times \{ b \} \xrightarrow{\ \ \simeq \ \ } \overline{B}^{(\gamma)}_{t = a}:= \operatorname{Spec} \left( R[t] / \bigl ( I_{\overline{B} ^{(\gamma)}}+(t-a) \bigr) \right). 
\]
\end{rmk}

\begin{defi}
We define a sheaf $L_{\overline{B}^{(\gamma)}_{1/d}}$ on $\overline{B}^{(\gamma)}_{1/d}$ by
\begin{align*}
& L_{\overline{B}^{(\gamma)}_{1/d}} := t^{w_\gamma({\bf f})} \left( \operatorname{det} ^{-1} \Bigl ( I_{\overline{B}_{k[t^{1/d}]}}/ I^2 _{\overline{B}_{k[t^{1/d}]}} \Bigr ) \right) \big|_{\overline{B}^{(\gamma)}_{1/d}} \otimes _{\mathcal{O}_{\overline{B}^{(\gamma)}_{1/d}}}
\Bigl ( \Omega_{\overline{A}^{(\gamma)}_{1/d}/k[t^{1/d}]} ^N \Bigr ) \big|_{\overline{B}^{(\gamma)}_{1/d}}, \\
&\text{where \quad} I_{\overline{B}_{k[t^{1/d}]}} :=  I_{\overline{B}}R[t^{1/d}], \quad w_\gamma({\bf f}) := w_{\gamma}(f_1) + \cdots + w_{\gamma}(f_c). 
\end{align*}
Here, $\operatorname{det} ^{-1} \Bigl ( I_{\overline{B}_{k[t^{1/d}]}}/ I^2_{\overline{B}_{k[t^{1/d}]}} \Bigr ) = \mathcal{H}om_{\mathcal{O}_{\overline{B}_{k[t^{1/d}]}}} \left( \bigwedge ^{c} \Bigl ( I_{\overline{B}_{k[t^{1/d}]}}/ I^2_{\overline{B}_{k[t^{1/d}]}} \Bigr ), \mathcal{O}_{\overline{B}_{k[t^{1/d}]}} \right)$.
\end{defi}

Recall that $r$ is a positive integer such that $\omega _{B}^{[r]}$ is invertible.

\begin{lem}\label{lem:L}
\begin{enumerate}
\item
There is a natural morphism $\Omega ^n _{\overline{B}^{(\gamma)}_{1/d}/k[t^{1/d}]} \to L_{\overline{B}^{(\gamma)}_{1/d}}$, and we have
\[
\operatorname{Im} \left ( \Omega ^n _{\overline{B}^{(\gamma)}_{1/d}/k[t^{1/d}]} \to L_{\overline{B}^{(\gamma)}_{1/d}} \right)  = \operatorname{Fitt}^n \Bigl ( \Omega _{\overline{B}^{(\gamma)}_{1/d}/k[t^{1/d}]} \Bigr ) \otimes _{\mathcal{O}_{\overline{B}^{(\gamma)}_{1/d}}} L_{\overline{B}^{(\gamma)}_{1/d}}.
\]  

\item
There is a natural morphism $t^{r \cdot w_{\gamma}({\bf f})} \Bigl ( \omega^{[r]} _{B_{k[t^{1/d}]}/k[t^{1/d}]} \big| _{\overline{B}^{(\gamma)}_{1/d}} \Bigr ) \to L_{\overline{B}^{(\gamma)}_{1/d}}^{\otimes r}$, 
and we have
\[
\operatorname{Im} \left (t^{r \cdot w_{\gamma}({\bf f})} \Bigl ( \omega^{[r]} _{B_{k[t^{1/d}]}/k[t^{1/d}]} \big| _{\overline{B}^{(\gamma)}_{1/d}} \Bigr ) \to L_{\overline{B}^{(\gamma)}_{1/d}}^{\otimes r} \right)
= t^{r \cdot \operatorname{age}(\gamma)} L_{\overline{B}^{(\gamma)}_{1/d}}^{\otimes r}. 
\]
\end{enumerate}
\end{lem}

\begin{proof}
We define an invertible sheaf $L_{\widehat{B}^{(\gamma)}_{1/d}}$ on $\widehat{B}^{(\gamma)}_{1/d}$ by
\[
L_{\widehat{B}^{(\gamma)}_{1/d}} := \left( \operatorname{det} ^{-1} \Bigl ( I_{\overline{B}_{k[t^{1/d}]}}/ I^2 _{\overline{B}_{k[t^{1/d}]}} \Bigr) \right) \big|_{\widehat{B}^{(\gamma)}_{1/d}} \otimes _{\mathcal{O}_{\widehat{B}^{(\gamma)}_{1/d}}}
\left( \Omega_{\overline{A}^{(\gamma)}_{1/d}/k[t^{1/d}]} ^N \right) \big|_{\widehat{B}^{(\gamma)}_{1/d}}. 
\]
Then, by the definitions of $L_{\widehat{B}^{(\gamma)}_{1/d}}$ and $L_{\overline{B}^{(\gamma)}_{1/d}}$, we have a natural injective morphism $L_{\overline{B}^{(\gamma)}_{1/d}} \hookrightarrow L_{\widehat{B}^{(\gamma)}_{1/d}} |_{\overline{B}^{(\gamma)}_{1/d}}$, 
and we have
\[
\tag{i}
\operatorname{Im} \left ( L_{\overline{B}^{(\gamma)}_{1/d}} \hookrightarrow L_{\widehat{B}^{(\gamma)}_{1/d}} |_{\overline{B}^{(\gamma)}_{1/d}} \right) = t^{w_\gamma({\bf f})} \left( L_{\widehat{B}^{(\gamma)}_{1/d}} |_{\overline{B}^{(\gamma)}_{1/d}} \right).
\]

Since $\overline{B} \to B$ is \'{e}tale in codimension one by Lemma \ref{lem:etaleness of B}(1), we have 
\begin{align*}
& \omega^{[r]} _{B_{k[t^{1/d}]}/k[t^{1/d}]} \big| _{\overline{B}_{k[t^{1/d}]}} \\
& \simeq \omega^{\otimes r} _{\overline{B}_{k[t^{1/d}]}/k[t^{1/d}]} \\
& \simeq \left( \operatorname{det} ^{-1} \Bigl( I_{\overline{B}_{k[t^{1/d}]}}/ I^2_{\overline{B}_{k[t^{1/d}]}} \Bigr) \right)^{\otimes r} \otimes _{\mathcal{O}_{\overline{B}_{k[t^{1/d}]}}} \Bigl ( \Omega_{\overline{A}_{k[t^{1/d}]}/k[t^{1/d}]} ^N \Bigr )^{\otimes r} \big|_{\overline{B}_{k[t^{1/d}]}}. 
\end{align*}
Furthermore, we have a natural morphism $\Omega_{\overline{A}_{k[t^{1/d}]}/k[t^{1/d}]} ^N \big|_{\overline{A}^{(\gamma)}_{1/d}} \to \Omega_{\overline{A}^{(\gamma)}_{1/d}/k[t^{1/d}]} ^N$, and its image is equal to $t^{\operatorname{age}(\gamma)} \left( \Omega_{\overline{A}^{(\gamma)}_{1/d}/k[t^{1/d}]} ^N \right)$. 
Hence, we have a natural morphism $\omega^{[r]} _{B_{k[t^{1/d}]}/k[t^{1/d}]} \big| _{\widehat{B}^{(\gamma)}_{1/d}} \to L_{\widehat{B}^{(\gamma)}_{1/d}}^{\otimes r}$, and its image is equal to
\[
\tag{ii}
\operatorname{Im} \left ( \omega^{[r]} _{B_{k[t^{1/d}]}/k[t^{1/d}]} \big| _{\widehat{B}^{(\gamma)}_{1/d}} \to L_{\widehat{B}^{(\gamma)}_{1/d}}^{\otimes r} \right) =
t^{r \cdot \operatorname{age}(\gamma)} L_{\widehat{B}^{(\gamma)}_{1/d}} ^{\otimes r}. 
\]
Then, the assertions in (2) follow from (i) and (ii). 

We shall prove (1) following the argument in \cite{NS22}*{Lemma 4.5}. 
It is sufficient to show that we have a natural morphism 
\[
\left( \bigwedge ^{c} \left( I_{\overline{B}_{k[t^{1/d}]}}/ I^2_{\overline{B}_{k[t^{1/d}]}} \right) \right) \big|_{\overline{B}^{(\gamma)}_{1/d}} \otimes _{\mathcal{O}_{\overline{B}^{(\gamma)}_{1/d}}} \Omega ^n _{\overline{B}^{(\gamma)}_{1/d}/k[t^{1/d}]} \to \left( \Omega_{\overline{A}^{(\gamma)}_{1/d}/k[t^{1/d}]} ^N \right) \big|_{\overline{B}^{(\gamma)}_{1/d}}, 
\]
and its image is equal to 
\[
t^{w_\gamma({\bf f})} \cdot \operatorname{Fitt}^n \left( \Omega _{\overline{B}^{(\gamma)}_{1/d}/k[t^{1/d}]} \right) \otimes \left( \Omega_{\overline{A}^{(\gamma)}_{1/d}/k[t^{1/d}]} ^N \right) \big|_{\overline{B}^{(\gamma)}_{1/d}}. 
\]
We set 
\begin{align*}
S_{\overline{B}_{k[t^{1/d}]}} &:= R[t^{1/d}] \big/ \left(  I_{\overline{B}}R[t^{1/d}] \right), \\
S_{\overline{B}^{(\gamma)}_{1/d}} &:= R[t^{1/d}] \big/ \left(  I_{\overline{B}^{(\gamma)}} R[t^{1/d}] \right), \\ 
I_{\overline{B}^{(\gamma)}_{1/d}} &:=  I_{\overline{B}^{(\gamma)}} R[t^{1/d}]. 
\end{align*}
Let $M \subset \Omega _{R[t^{1/d}]/k[t^{1/d}]}$ be the $R[t^{1/d}]$-submodule generated by $df$ for $f \in I_{\overline{B}^{(\gamma)}_{1/d}}$. 
Then, we have 
\[
\left( \Omega _{R[t^{1/d}]/k[t^{1/d}]} \big/M \right) \otimes _{R[t^{1/d}]} S_{\overline{B}^{(\gamma)}_{1/d}} \simeq  \Omega _{S_{\overline{B}^{(\gamma)}_{1/d}}/k[t^{1/d}]}.
\]

Since $I_{\overline{B}^{(\gamma)}_{1/d}}$ is generated by $c$ elements, we have $\bigwedge ^{c+1} M = 0$. 
Therefore, a natural map $\Omega _{R[t^{1/d}]/k[t^{1/d}]} ^{c} \otimes \Omega _{R[t^{1/d}]/k[t^{1/d}]} ^{N-c} \to \Omega _{R[t^{1/d}]/k[t^{1/d}]} ^N$ induces a map 
\[
\bigwedge ^{c} M \otimes _{R[t^{1/d}]} \bigwedge ^{N-c} \left( \Omega _{R[t^{1/d}]/k[t^{1/d}]} \big/M \right) \to \Omega _{R[t^{1/d}]/k[t^{1/d}]} ^N. 
\] 
Furthermore, we have natural maps
\[
\left( I_{\overline{B}_{k[t^{1/d}]}}/ I^2_{\overline{B}_{k[t^{1/d}]}} \right) \otimes _{S_{\overline{B}_{k[t^{1/d}]}}} S_{\overline{B}^{(\gamma)}_{1/d}} \to I_{\overline{B}^{(\gamma)}_{1/d}}/ I^2 _{\overline{B}^{(\gamma)}_{1/d}} \xrightarrow{\overline{h} \mapsto d(h) \otimes 1} M \otimes _{R[t^{1/d}]} S_{\overline{B}^{(\gamma)}_{1/d}}. 
\]
Therefore, we have a natural map
\begin{align*}
\tag{iii}
\left( \bigwedge ^{c} \left( I_{\overline{B}_{k[t^{1/d}]}}/ I^2_{\overline{B}_{k[t^{1/d}]}} \right)  \otimes _{S_{\overline{B}_{k[t^{1/d}]}}} S_{\overline{B}^{(\gamma)}_{1/d}} \right) & \otimes _{S_{\overline{B}^{(\gamma)}_{1/d}}} \Omega _{S_{\overline{B}^{(\gamma)}_{1/d}}/k[t^{1/d}]} ^{N-c} \\
 \longrightarrow & \ \Omega _{R[t^{1/d}]/k[t^{1/d}]} ^N \otimes _{R[t^{1/d}]} S_{\overline{B}^{(\gamma)}_{1/d}}. 
\end{align*}

The $S_{\overline{B}^{(\gamma)}_{1/d}}$-module 
\[
\left( \bigwedge ^{c} \left( I_{\overline{B}_{k[t^{1/d}]}}/ I^2_{\overline{B}_{k[t^{1/d}]}} \right) \otimes _{S_{\overline{B}_{k[t^{1/d}]}}} S_{\overline{B}^{(\gamma)}_{1/d}} \right) \otimes _{S_{\overline{B}^{(\gamma)}_{1/d}}} \Omega _{S_{\overline{B}^{(\gamma)}_{1/d}}/k[t^{1/d}]} ^{N-c}
\]
is generated by the elements of the form 
\[
\left( \overline{f_1} \wedge \cdots \wedge \overline{f_c} \right) \otimes \left( d \overline{x}_{i _1} \wedge \cdots \wedge d \overline{x}_{i_{N-c}} \right). 
\]
Its image by the map (iii) is
\begin{align*}
&d \left( \lambda _{\gamma} ^* (f_1) \right) \wedge \cdots \wedge d \left( \lambda _{\gamma} ^* (f_c) \right) \wedge  d x_{i _1} \wedge \cdots \wedge d x_{i_{N-c}} \\
&=t^{w_{\gamma}({\bf f})} \cdot d \left( \lambda _{\gamma} ^* (f_1) t^{- w_{\gamma} (f_1)} \right) \wedge \cdots \wedge d \left( \lambda _{\gamma} ^* (f_c) t^{- w_{\gamma} (f_c)} \right) \wedge d x_{i _1} \wedge \cdots \wedge dx_{i_{N-c}}. 
\end{align*}
This is equal to 
\[
\pm t^{w_{\gamma}({\bf f})} \cdot \Delta \cdot dx_1 \wedge \cdots \wedge dx_N, 
\]
where $\Delta$ is the determinant of the Jacobian matrix with respect to $\lambda _{\gamma} ^* (f_i) t^{- w_{\gamma} (f_i)}$ for $1 \le i \le c$ and $\partial x_j$ for $j \in \{ 1, \ldots ,N \} \setminus \{ i_1, \ldots, i_{N-c} \}$. 
Therefore, we conclude that the image of the map (iii) is equal to 
\[
t^{w_\gamma({\bf f})} \cdot \operatorname{Fitt}^n \Bigl ( \Omega _{\overline{B}^{(\gamma)}_{1/d}/k[t^{1/d}]} \Bigr) \Bigl ( \Omega _{R[t^{1/d}]/k[t^{1/d}]} ^N \otimes _{R[t^{1/d}]} S_{\overline{B}^{(\gamma)}_{1/d}} \Bigr), 
\]
which completes the proof of (1). 
\end{proof}

\begin{defi}
By composing $\Omega^n _{B_{k[t^{1/d}]}/k[t^{1/d}]} \big|_{\overline{B}^{(\gamma)}_{1/d}} \to \Omega ^n _{\overline{B}^{(\gamma)}_{1/d}/k[t^{1/d}]}$ and $\Omega ^n _{\overline{B}^{(\gamma)}_{1/d}/k[t^{1/d}]} \to L_{\overline{B}^{(\gamma)}_{1/d}}$ in Lemma \ref{lem:L}(1), 
we obtain a natural morphism $\Omega^n _{B_{k[t^{1/d}]}/k[t^{1/d}]} \big|_{\overline{B}^{(\gamma)}_{1/d}} \to L_{\overline{B}^{(\gamma)}_{1/d}}$. 
Then since $L_{\overline{B}^{(\gamma)}_{1/d}}$ is generated by one global section, we define an ideal sheaf 
$\mathfrak{n}'_{\overline{B}^{(\gamma)}_{1/d}} \subset \mathcal{O}_{\overline{B}^{(\gamma)}_{1/d}}$ by
\[
\operatorname{Im} \left( 
\Omega^n _{B_{k[t^{1/d}]}/k[t^{1/d}]} \big|_{\overline{B}^{(\gamma)}_{1/d}} \to L_{\overline{B}^{(\gamma)}_{1/d}}  
\right) = \mathfrak{n}'_{\overline{B}^{(\gamma)}_{1/d}} \otimes _{\mathcal{O}_{\overline{B}^{(\gamma)}_{1/d}}} L_{\overline{B}^{(\gamma)}_{1/d}}. 
\]
\end{defi}

\begin{lem}\label{lem:L2}
We have the following equality of ideal sheaves on $\overline{B}^{(\gamma)}_{1/d}$: 
\[
t^{r \cdot \operatorname{age}(\gamma)} \Bigl( \mathfrak{n}_{r,B} \mathcal{O}_{\overline{B}^{(\gamma)}_{1/d}} \Bigr) = t^{r \cdot w_{\gamma}({\bf f})} \Bigl( \mathfrak{n}'_{\overline{B}^{(\gamma)}_{1/d}} \Bigr)^r. 
\]
\end{lem}
\begin{proof}
By the definition of $\mathfrak{n}'_{\overline{B}^{(\gamma)}_{1/d}}$, we have 
\[
\operatorname{Im} \left( 
t^{r \cdot w_{\gamma}({\bf f})} \Bigl (\Omega^n _{B_{k[t^{1/d}]}/k[t^{1/d}]} \Bigr)^{\otimes r} \big|_{\overline{B}^{(\gamma)}_{1/d}} \to L_{\overline{B}^{(\gamma)}_{1/d}} ^{\otimes r}
\right) = t^{r \cdot w_{\gamma}({\bf f})} \Bigl( \mathfrak{n}'_{\overline{B}^{(\gamma)}_{1/d}} \Bigr)^r \otimes L_{\overline{B}^{(\gamma)}_{1/d}} ^{\otimes r}. 
\]
On the other hand, by the definition of $\mathfrak{n}_{r,B}$ and Remark \ref{rmk:nash}, we have 
\begin{align*}
\operatorname{Im} & \left( 
t^{r \cdot w_{\gamma}({\bf f})} \Bigl (\Omega^n _{B_{k[t^{1/d}]}/k[t^{1/d}]} \Bigr)^{\otimes r} \big|_{\overline{B}^{(\gamma)}_{1/d}} 
	\to \omega _{B_{k[t^{1/d}]}/k[t^{1/d}]} ^{[r]} \big|_{\overline{B}^{(\gamma)}_{1/d}}
\right) \\
&= t^{r \cdot w_{\gamma}({\bf f})} \Bigl( \mathfrak{n}_{r,B} \mathcal{O}_{\overline{B}^{(\gamma)}_{1/d}} \Bigr) \otimes \omega _{B_{k[t^{1/d}]}/k[t^{1/d}]} ^{[r]} \big|_{\overline{B}^{(\gamma)}_{1/d}}. 
\end{align*}
Furthermore, by Lemma \ref{lem:L}(2), we have 
\[
\operatorname{Im} \left(
t^{r \cdot w_{\gamma}({\bf f})}  \Bigl( \omega _{B_{k[t^{1/d}]}/k[t^{1/d}]} ^{[r]} \big|_{\overline{B}^{(\gamma)}_{1/d}} \Bigr)
\to L_{\overline{B}^{(\gamma)}_{1/d}} ^{\otimes r}
\right) 
= t^{r \cdot \operatorname{age}(\gamma)} L_{\overline{B}^{(\gamma)}_{1/d}} ^{\otimes r}. 
\]
Therefore, we have 
\begin{align*}
\operatorname{Im} & \left( 
t^{r \cdot w_{\gamma}({\bf f})} \Bigl ( \Omega^n _{B_{k[t^{1/d}]}/k[t^{1/d}]} \Bigr) ^{\otimes r} \big|_{\overline{B}^{(\gamma)}_{1/d}} \to L_{\overline{B}^{(\gamma)}_{1/d}} ^{\otimes r} \right) \\ 
& = t^{r \cdot \operatorname{age}(\gamma)} \Bigl( \mathfrak{n}_{r,B} \mathcal{O}_{\overline{B}^{(\gamma)}_{1/d}} \Bigr) \otimes L_{\overline{B}^{(\gamma)}_{1/d}} ^{\otimes r}, 
\end{align*}
which proves the desired equality. 
\end{proof}

\begin{lem}[cf.\ \cite{NS22}*{Lemmas 4.6(2) and 4.7}]\label{lem:L3}
Let $\alpha \in \overline{B}^{(\gamma)} _{\infty}$ be an arc with 
$\operatorname{ord}_{\alpha} \left( \operatorname{Jac} _{\overline{B}^{(\gamma)}/k[t]} \right) < \infty$. 
Then we have
\[\operatorname{ord}_{\alpha} \left( \operatorname{jac}_{\overline{\mu} _{\gamma}} \right) 
+ \operatorname{ord}_{\alpha} \left( \operatorname{Jac} _{\overline{B}^{(\gamma)}/k[t]} \right)
= \frac{1}{r} \operatorname{ord}_{\alpha} \left( \mathfrak{n}_{r,B} \mathcal{O}_{\overline{B}^{(\gamma)}} \right) 
+ \operatorname{age}(\gamma)-w_\gamma({\bf f}). 
\]
\end{lem}
\begin{proof}
Note that $\overline{B}^{(\gamma)} \times _{\operatorname{Spec} k[t]} \operatorname{Spec} k[t^{1/d}] = \overline{B}^{(\gamma)}_{1/d}$. 
Let $\beta:\operatorname{Spec} k[[t^{1/d}]]\rightarrow \overline{B}_{1/d}^{(\gamma)}$ be the unique lift of $\alpha: \operatorname{Spec} k[[t]]\rightarrow\overline{B}^{(\gamma)}$ induced by the base change. 
For an ideal sheaf $\mathfrak{a} \subset \mathcal{O}_{\overline{B}^{(\gamma)}_{1/d}}$, we define $\operatorname{ord} _{\beta} (\mathfrak{a})$ as follows: 
For the corresponding ring homomorphism $\beta ^*: \mathcal{O}_{\overline{B}^{(\gamma)}_{1/d}} \to k[[t^{1/d}]]$, 
we have $\beta ^* (\mathfrak{a})k[[t^{1/d}]] = \bigl( t^{\frac{\ell}{d}} \bigr)$ for some $\ell \in \mathbb{Z}_{\ge 0}$. 
Then, we define $\operatorname{ord} _{\beta} (\mathfrak{a}) := \frac{\ell}{d}$. 

Since $\operatorname{Fitt}^n \Bigl( \Omega _{\overline{B}^{(\gamma)}_{1/d}/k[t^{1/d}]} \Bigr) = \operatorname{Jac} _{\overline{B}^{(\gamma)}/k[t]} \mathcal{O}_{\overline{B}^{(\gamma)}_{1/d}}$, we have 
\[
\operatorname{ord}_{\beta} \left( \operatorname{Fitt}^n \Bigl( \Omega _{\overline{B}^{(\gamma)}_{1/d}/k[t^{1/d}]} \Bigr) \right)
= \operatorname{ord}_{\alpha} \left( \operatorname{Jac} _{\overline{B}^{(\gamma)}/k[t]} \right) < \infty. 
\]
Therefore, 
$\overline{B}^{(\gamma)}_{1/d}$ is smooth over $k[t^{1/d}]$ at the point $\beta (\eta)$, where $\eta$ denotes the generic point of  $\operatorname{Spec}(k[[t^{1/d}]])$. 
Hence we have an isomorphism of $k[[t^{1/d}]]$-modules:
\[
\beta ^* \Omega _{\overline{B}^{(\gamma)}_{1/d}/k[t^{1/d}]} \big/ \Bigl( \beta ^* \Omega _{\overline{B}^{(\gamma)}_{1/d}/k[t^{1/d}]} \Bigr)_{\rm tor} \simeq k[[t^{1/d}]]^{\oplus n}. \tag{i}
\]

Let $a$ be the length of the $k[[t^{1/d}]]$-module
\[
\operatorname{Coker} \left( \beta ^* \Bigl( \Omega _{B_{k[t^{1/d}]}/k[t^{1/d}]} \big|_{\overline{B}^{(\gamma)}_{1/d}} \Bigr) \to \beta ^* \Omega _{\overline{B}^{(\gamma)}_{1/d} /k[t^{1/d}]} \big/ \Bigl( \beta ^* \Omega _{\overline{B}^{(\gamma)}_{1/d} /k[t^{1/d}]} \Bigr)_{\rm tor} \right).
\]
Taking the $n$-th exterior power, the isomorphism (i) implies that $a$ is equal to the length of 
\[
\operatorname{Coker} \left( \beta ^* \Bigl( \Omega^n _{B_{k[t^{1/d}]}/k[t^{1/d}]} \big|_{\overline{B}^{(\gamma)}_{1/d}} \Bigr) \to \beta ^* \Omega^n _{\overline{B}^{(\gamma)}_{1/d} /k[t^{1/d}]} \big/ \Bigl( \beta ^* \Omega^n _{\overline{B}^{(\gamma)}_{1/d} / k[t^{1/d}]} \Bigr)_{\rm tor} \right). \tag{ii}
\]
We shall prove 
\[
\frac{a}{d} + 
\operatorname{ord}_{\beta} \left( \operatorname{Fitt}^n \Bigl( \Omega _{\overline{B}^{(\gamma)}_{1/d}/k[t^{1/d}]} \Bigr) \right)
= \operatorname{ord}_{\beta} \Bigl( \mathfrak{n}'_{\overline{B}^{(\gamma)}_{1/d}} \Bigr). \tag{iii}
\]
By the definition of $\mathfrak{n}'_{\overline{B}^{(\gamma)}_{1/d}}$, we have 
\[
\operatorname{ord}_{\beta} \Bigl( \mathfrak{n}'_{\overline{B}^{(\gamma)}_{1/d}} \Bigr) 
= \frac{1}{d} \operatorname{length} \left( \operatorname{Coker} \Bigl( \beta^* \Bigl( \Omega^n _{B_{k[t^{1/d}]}/k[t^{1/d}]} \big|_{\overline{B}^{(\gamma)}_{1/d}} \Bigr ) \to \beta ^* L_{\overline{B}^{(\gamma)}_{1/d}} \Bigr) \right). \tag{iv}
\]
By Lemma \ref{lem:L}(1),  we have
\[
\operatorname{ord}_{\beta} \left( \operatorname{Fitt}^n \Bigl( \Omega _{\overline{B}^{(\gamma)}_{1/d}/k[t^{1/d}]} \Bigr) \right)
= 
\frac{1}{d} \operatorname{length} \left( \operatorname{Coker} \Bigl( \beta ^* \Omega^n _{\overline{B}^{(\gamma)}_{1/d} /k[t^{1/d}]} \to \beta ^* L_{\overline{B}^{(\gamma)}_{1/d}} \Bigr) \right ) .
\]
Since $L_{\overline{B}^{(\gamma)}_{1/d}}$ is torsion-free, the map $\beta ^* \Omega^n _{\overline{B}^{(\gamma)}_{1/d} /k[t^{1/d}]} \to \beta ^* L_{\overline{B}^{(\gamma)}_{1/d}}$ factors through the torsion-free quotient
\[
\beta ^* \Omega^n _{\overline{B}^{(\gamma)}_{1/d} /k[t^{1/d}]} \big/ \Bigl( \beta ^* \Omega^n _{\overline{B}^{(\gamma)}_{1/d} /k[t^{1/d}]} \Bigr)_{\rm tor}. 
\]
Therefore, $\operatorname{ord}_{\beta} \left( \operatorname{Fitt}^n \Bigl( \Omega _{\overline{B}^{(\gamma)}_{1/d}/k[t^{1/d}]} \Bigr) \right)$ is equal to 
\[
\frac{1}{d} \operatorname{length} \left( \operatorname{Coker} \Bigl( \beta ^* \Omega^n _{\overline{B}^{(\gamma)}_{1/d} /k[t^{1/d}]} \big/ \Bigl( \beta ^* \Omega^n _{\overline{B}^{(\gamma)}_{1/d} /k[t^{1/d}]} \Bigr)_{\rm tor} \to \beta ^* L_{\overline{B}^{(\gamma)}_{1/d}}  \Bigr) \right ) . \tag{v}
\]
Thus, the assertion (iii) follows from (ii), (iv), and (v).

Since 
\begin{align*}
&\operatorname{Coker} \left( \beta ^* \Bigl( \Omega _{B_{k[t^{1/d}]}/k[t^{1/d}]} \big|_{\overline{B}^{(\gamma)}_{1/d}} \Bigr) \to \Bigl ( \beta ^* \Omega _{\overline{B}^{(\gamma)}_{1/d} /k[t^{1/d}]} \Bigr ) \big/ \Bigl ( \beta ^* \Omega _{\overline{B}^{(\gamma)}_{1/d} /k[t^{1/d}]} \Bigr )_{\rm tor} \right) \\ 
&\simeq 
\operatorname{Coker} \left( \alpha ^* \left ( \Omega _{B \times \mathbb{A}_k ^1/k[t]} \big|_{\overline{B}^{(\gamma)}} \right) \to 
\left( \alpha ^* \Omega _{\overline{B}^{(\gamma)}/k[t]} \right) \big/ \left( \alpha ^* \Omega _{\overline{B}^{(\gamma)}/k[t]} \right)_{\rm tor} \right) \otimes_{k[[t]]} k[[t^{\frac{1}{d}}]], 
\end{align*}
we have $\frac{a}{d} = \operatorname{ord}_{\alpha} \bigl( \operatorname{jac}_{\overline{\mu} _{\gamma}} \bigr)$ .
Therefore, the assertion follows from (iii) and Lemma \ref{lem:L2}. 
\end{proof}

We define a $k[t]$-scheme $\widetilde{B}^{(\gamma)}_{1/d}$ by 
\begin{align*}
\widetilde{B}^{(\gamma)}_{1/d}
&:=  \operatorname{Spec} \left( R^{C_{\gamma}}[t^{1/d}] \big/ \left(  I_{\overline{B}^{(\gamma)}} \cap R^{C_{\gamma}}[t] \right) R^{C_{\gamma}}[t^{1/d}] \right) \\
&=  \widetilde{B}^{(\gamma)} \times _{\operatorname{Spec} k[t]} \operatorname{Spec} k[t^{1/d}].
\end{align*}
Then, we have a natural morphism $\overline{B}^{(\gamma)}_{1/d} \to \widetilde{B}^{(\gamma)}_{1/d}$. 

\begin{defi}
By composing $\Omega^n _{\widetilde{B}^{(\gamma)}_{1/d}/k[t^{1/d}]} \big|_{\overline{B}^{(\gamma)}_{1/d}} \to \Omega ^n _{\overline{B}^{(\gamma)}_{1/d}/k[t^{1/d}]}$ and $\Omega ^n _{\overline{B}^{(\gamma)}_{1/d}/k[t^{1/d}]} \to L_{\overline{B}^{(\gamma)}_{1/d}}$ in Lemma \ref{lem:L}(1), 
we obtain a natural morphism $\Omega^n _{\widetilde{B}^{(\gamma)}_{1/d}/k[t^{1/d}]} \big|_{\overline{B}^{(\gamma)}_{1/d}} \to L_{\overline{B}^{(\gamma)}_{1/d}}$. 
Then, we denote by $\mathfrak{n}''_{\overline{B}^{(\gamma)}_{1/d}}$ the ideal sheaf on $\overline{B}^{(\gamma)}_{1/d}$ satisfying
\[
\operatorname{Im} \left( \Omega^n _{\widetilde{B}^{(\gamma)}_{1/d}/k[t^{1/d}]} \big|_{\overline{B}^{(\gamma)}_{1/d}} \to L_{\overline{B}^{(\gamma)}_{1/d}} \right) 
= \mathfrak{n}''_{\overline{B}^{(\gamma)}_{1/d}} \otimes _{\mathcal{O}_{\overline{B}^{(\gamma)}_{1/d}}} L_{\overline{B}^{(\gamma)}_{1/d}}. 
\]
\end{defi}

\begin{lem}\label{lem:L4}
There exists a $C_{\gamma}$-invariant ideal sheaf $\mathfrak{n}'_{p}$ on $\overline{B}^{(\gamma)}$ such that 
$\mathfrak{n}''_{\overline{B}^{(\gamma)}_{1/d}} = \mathfrak{n}'_{p} \mathcal{O}_{\overline{B}^{(\gamma)}_{1/d}}$. 
\end{lem}
\begin{proof}
Take elements $g_1, \ldots , g_{\ell} \in R^{C_{\gamma}}$ satisfying $R^{C_{\gamma}} = k[g_1, \ldots, g_{\ell}]$. 
For a subset $J \subset \{ 1, \ldots, \ell \}$ with $\# J = N-c$, we define $\Delta _J \in R[t]$ as the determinant of the Jacobian matrix (of size $N$) with respect to $\lambda ^* _{\gamma} (f_i) t^{- w_{\gamma}(f_i)}$ for $1 \le i \le c$ and $g_i$ for $i \in J$, and $\partial x_j$ for $1 \le j \le N$. 
Note that $\Delta _J$ is defined up to sign.
We define an ideal $\mathfrak{n}'_{p} \subset R[t] / I_{\overline{B}^{(\gamma)}}$ as 
\[
\mathfrak{n}'_{p} := \bigl ( \overline{ \Delta _J} \ \big| \ \text{$J \subset \{ 1, \ldots, \ell \}$ with $\# J = N-c$} \bigr ), 
\]
where $\overline{ \Delta _J}$ denoted the image of $\Delta _J \in R[t]$ in $R[t] / I_{\overline{B}^{(\gamma)}}$. 
For any $\delta \in C_{\gamma}$, by Lemma \ref{lem:lambda}(1), we have 
\[
\delta \left( \lambda ^* _{\gamma} (f_i) t^{- w_{\gamma}(f_i)} \right) 
= 
\xi^{d w_{\delta} (f_i)} \lambda ^* _{\gamma} (f_i) t^{- w_{\gamma}(f_i)}. 
\]
Furthermore, note that $\delta$ maps $x_i$'s to their $k$-linear independent $k$-linear combinations. 
Therefore, for each $J$, we have 
\begin{align*}
&
\xi^{d w_{\delta} ({\bf f})} \cdot \Delta _J \cdot dx_1 \wedge \cdots \wedge dx_N \\
&=
\xi^{d w_{\delta} ({\bf f})} \cdot 
\bigl( \wedge_{i \in J} \ d g_i \bigr) 
\wedge 
\left( \wedge _{1 \le i \le c}\ d \left( \lambda ^* _{\gamma} (f_i) t^{- w_{\gamma}(f_i)} \right) \right) \\
&=
\bigl( \wedge_{i \in J}\ d (\delta (g_i)) \bigr) 
\wedge 
\left( \wedge _{1 \le i \le c} \ d \left( \delta \left( \lambda ^* _{\gamma} (f_i) t^{- w_{\gamma}(f_i)} \right) \right) \right) \\
&=
\delta \left( 
\bigl( \wedge_{i \in J}\ d g_i \bigr) 
\wedge 
\left( \wedge _{1 \le i \le c} \ d \left( \lambda ^* _{\gamma} (f_i) t^{- w_{\gamma}(f_i)} \right) \right)
\right)\\
&=
\delta \left(
\Delta _J \cdot dx_1 \wedge \cdots \wedge dx_N
\right) \\
&=
\delta \left( \Delta _J \right) 
\cdot d (\delta (x_1)) \wedge \cdots \wedge d (\delta (x_N)) \\
&= 
\delta \left( \Delta _J \right) \cdot s \cdot dx_1 \wedge \cdots \wedge dx_N 
\end{align*}
for some $s \in k^{\times}$. 
Hence, we have $\delta (\Delta _J) = s' \cdot \Delta _J$ for some $s' \in k$. 
Hence, we conclude that $\mathfrak{n}'_{p}$ is $C_{\gamma}$-invariant. 

We shall prove that this $\mathfrak{n}'_{p}$ satisfies the assertion $\mathfrak{n}''_{\overline{B}^{(\gamma)}_{1/d}} = \mathfrak{n}'_{p} \mathcal{O}_{\overline{B}^{(\gamma)}_{1/d}}$.  
We set 
\[
S_{\widetilde{B}^{(\gamma)}_{1/d}} := R^{C_{\gamma}}[t^{1/d}] \big/ \left(  I_{\overline{B}^{(\gamma)}} \cap R^{C_{\gamma}}[t] \right) R^{C_{\gamma}}[t^{1/d}]. 
\]
We also use the notation in the proof of Lemma \ref{lem:L}. 
By the proof of Lemma \ref{lem:L}(1), the considered map $\Omega^n _{\widetilde{B}^{(\gamma)}_{1/d}/k[t^{1/d}]} \big|_{\overline{B}^{(\gamma)}_{1/d}} \to L_{\overline{B}^{(\gamma)}_{1/d}}$ is corresponding to 
\begin{align*}
\tag{i}
\left ( \left( \bigwedge ^{c} \left( I_{\overline{B}_{k[t^{1/d}]}}/ I^2_{\overline{B}_{k[t^{1/d}]}} \right) \otimes _{S_{\overline{B}_{k[t^{1/d}]}}} S_{\widetilde{B}^{(\gamma)}_{1/d}} \right)  \otimes _{S_{\widetilde{B}^{(\gamma)}_{1/d}}} \Omega _{S_{\widetilde{B}^{(\gamma)}_{1/d}}/k[t^{1/d}]} ^{N-c} \right) &
\otimes _{S_{\widetilde{B}^{(\gamma)}_{1/d}}} S_{\overline{B}^{(\gamma)}_{1/d}} \\
\longrightarrow \ \ \Omega _{R[t^{1/d}]/k[t^{1/d}]} ^N \otimes _{R[t^{1/d}]} S_{\overline{B}^{(\gamma)}_{1/d}}. & 
\end{align*}

The
$S_{\widetilde{B}^{(\gamma)}_{1/d}}$-module 
\[
\left( \bigwedge ^{c} \left( I_{\overline{B}_{k[t^{1/d}]}}/ I^2_{\overline{B}_{k[t^{1/d}]}} \right) \otimes _{S_{\overline{B}_{k[t^{1/d}]}}} S_{\widetilde{B}^{(\gamma)}_{1/d}} \right) \otimes _{S_{\widetilde{B}^{(\gamma)}_{1/d}}} \Omega _{S_{\widetilde{B}^{(\gamma)}_{1/d}}/k[t^{1/d}]} ^{N-c}
\] is generated by the elements of the form 
\[
\left( \overline{f_1} \wedge \cdots \wedge \overline{f_c} \right) \otimes \left( d \overline{g}_{i_1} \wedge \cdots \wedge d \overline{g}_{i_{N-c}} \right).  
\]
Its image by the map (i) is
\begin{align*}
&d \left( \lambda _{\gamma} ^* (f_1) \right) \wedge \cdots \wedge d \left( \lambda _{\gamma} ^* (f_c) \right) \wedge  dg_{i _1} \wedge \cdots \wedge dg_{i_{N-c}} \\
&=t^{w_{\gamma}({\bf f})} \cdot d \left( \lambda _{\gamma} ^* (f_1) t^{- w_{\gamma} (f_1)} \right) \wedge \cdots \wedge d \left( \lambda _{\gamma} ^* (f_c) t^{- w_{\gamma} (f_c)} \right) \wedge  d g_{i _1} \wedge \cdots \wedge d g_{i_{N-c}}. 
\end{align*}
This is equal to 
\[
t^{w_{\gamma}({\bf f})} \cdot \Delta _J \cdot dx_1 \wedge \cdots \wedge dx_N
\]
for $J = \{ i_1, \ldots , i_{N-c} \}$. 
Therefore, we conclude that the image of the map (i) is equal to 
\[
t^{w_\gamma({\bf f})} \mathfrak{n}'_{p} \left( \Omega _{R[t^{1/d}]/k[t^{1/d}]} ^N \otimes _{R[t^{1/d}]} S_{\overline{B}^{(\gamma)}_{1/d}} \right), 
\]
which completes the proof. 
\end{proof}

\begin{lem}[cf.\ \cite{NS22}*{Lemma 4.6(1)}]\label{lem:L5}
Let $\alpha \in \overline{B}^{(\gamma)} _{\infty}$ be an arc with 
$\operatorname{ord}_{\alpha} \left( \operatorname{Jac} _{\overline{B}^{(\gamma)}/k[t]} \right) < \infty$. 
Then, for an ideal sheaf $\mathfrak{n}'_{p}$ in Lemma \ref{lem:L4}, we have
\[
\operatorname{ord}_{\alpha} \left( \operatorname{jac}_p \right) + \operatorname{ord}_{\alpha} \left( \operatorname{Jac} _{\overline{B}^{(\gamma)}/k[t]} \right)
= \operatorname{ord}_{\alpha} \left( \mathfrak{n}'_{p} \right). 
\]
\end{lem}
\begin{proof}
The assertion is proved by the same argument as in Lemma \ref{lem:L3}. 
\end{proof}

\section{Minimal log discrepancies and arc spaces for quotient varieties}\label{section:mld}
We keep the notation from Sections \ref{section:setting} and \ref{section:order}. 
We fix an ideal sheaf $\mathfrak{n}'_{p}$ on $\overline{B}^{(\gamma)}$ satisfying the statement of Lemma \ref{lem:L4}. 
Let $I_Z \subset \mathcal{O}_B$ be the ideal sheaf on $B$ defining the closed subscheme $(B \cap Z)_{\rm red} \subset B$. 
Let $x \in A = \overline{A}/G$ be the image of the origin of $\overline{A} = \mathbb{A}^N _k$. 
Then, we have $x \in B$ since $f_1, \ldots, f_c \in (x_1, \ldots, x_N)$. 
Let $\mathfrak{m}_x \subset \mathcal{O}_B$ denote the corresponding maximal ideal. 

In this section and the next section, we will consider the codimensions of cylinders in $\overline{B}^{(\gamma)} _{\infty}$ with respect to $n$, 
and consider the codimensions of cylinders in $\overline{A}^{(\gamma)} _{\infty}$ with respect to $N$, respectively (cf.\ Remark \ref{rmk:seisitu}(2)). 

\begin{lem}\label{lem:thin}
Let $\mathfrak{a} \subset \mathcal{O}_B$ be an ideal sheaf on $B$. 
Let $J$ be one of the following ideal sheaves on $\overline{B}^{(\gamma)}$: 
\[
\mathfrak{a} \mathcal{O}_{\overline{B}^{(\gamma)}}, \quad
\operatorname{Jac}_{\overline{B}^{(\gamma)}/k[t]}, \quad 
\operatorname{Jac}_{\widetilde{B}^{(\gamma)}/k[t]} \mathcal{O}_{\overline{B}^{(\gamma)}}, \quad 
\mathfrak{n}'_{p}.
\]
Let $W$ be the subscheme of $\overline{B}^{(\gamma)}$ defined by $J$. 
Then the following assertions hold. 

\begin{enumerate}
\item 
The ideal $J$ is $C_{\gamma}$-invariant. 

\item 
Suppose $\mathfrak{a} \not = 0$. 
Then, $W_{\infty}$ is a thin set of $\overline{B}^{(\gamma)} _{\infty}$.

\item
Suppose $\mathfrak{a} \not = 0$. 
Then, $\overline{\mu}_{\gamma \infty} (W_{\infty})$ is a thin set of $B_{\infty}$.
\end{enumerate}
In (3), we consider $B_{\infty}$ to be the arc space of $B \times _{\operatorname{Spec} k} \operatorname{Spec} k[t]$ as a $k[t]$-schemes, 
and we adopt Definition \ref{defi:thin} for the definition of thin sets.
\end{lem}
\begin{proof}
Note that the map $\overline{B}^{(\gamma)} \to B$ factors $\overline{B}^{(\gamma)} \to \widetilde{B}^{(\gamma)} \to B$, and the first morphism $\overline{B}^{(\gamma)} \to \widetilde{B}^{(\gamma)}$ is the quotient morphism by $C_{\gamma}$ (Remark \ref{rmk:seisitu}(1)). 
Therefore, $\mathfrak{a} \mathcal{O}_{\overline{B}^{(\gamma)}}$ and $\operatorname{Jac}_{\widetilde{B}^{(\gamma)}/k[t]} \mathcal{O}_{\overline{B}^{(\gamma)}}$ are $C_{\gamma}$-invariant. The assertion (1) for the other cases follows from the definition of $\operatorname{Jac}_{\overline{B}^{(\gamma)}/k[t]}$ and the choice of $\mathfrak{n}'_{p}$ in Lemma \ref{lem:L4}. 

We shall prove (2). 
Let $X$ be the unique component of $\overline{B}^{(\gamma)}$ which dominates $\operatorname{Spec} k[t]$ (Remark \ref{rmk:seisitu}(3)). 
Since $\overline{B}^{(\gamma)} _{\infty} = X_{\infty}$ by Remark \ref{rmk:seisitu}(3), we have 
$W_{\infty} = (W \cap X)_{\infty}$. 
Therefore, in order to prove that $W_{\infty}$ is a thin set of $\overline{B}^{(\gamma)} _{\infty}$, it is sufficient to show $\dim (W \cap X) \le n$. 
Note that $\dim X = n+1$ (Remark \ref{rmk:seisitu}(3)) and that
$X_{t \not = 0} = \overline{B}^{(\gamma)}_{t \not = 0}$ is integral (Remark \ref{rmk:seisitu}(3)). 
Therefore, it is sufficient to show that $J \mathcal{O}_{X_{t=a}} \not = 0$ for some $a \in k^{\times}$. 

By the identification 
\[
\overline{B} \simeq \overline{B} \times \{ 1 \} \xrightarrow{\ \ \simeq \ \ } \overline{B}^{(\gamma)}_{t = 1} = X_{t = 1}. 
\]
in Remark \ref{rmk:seisitu2}(i), we have 
\[
\mathfrak{a} \mathcal{O}_{X_{t=1}} = \mathfrak{a}, \quad
\operatorname{Jac}_{\overline{B}^{(\gamma)}/k[t]} \mathcal{O}_{X_{t=1}} = \operatorname{Jac}_{\overline{B}}, \quad
\operatorname{Jac}_{\widetilde{B}^{(\gamma)}/k[t]} \mathcal{O}_{X_{t=1}} = 
\operatorname{Jac}_{\overline{B}/{C_{\gamma}}} \mathcal{O}_{\overline{B}}. 
\]
Since these are non-zero ideal sheaves on $\overline{B}$, we complete the proof of (2) except for the case $J = \mathfrak{n}'_{p}$. 

We prove the assertion (2) for the case $J = \mathfrak{n}'_{p}$. 
We set 
\[
J_1 := \operatorname{Jac}_{\overline{B}^{(\gamma)}/k[t]}, \quad
J_2 := \operatorname{Jac}_{\widetilde{B}^{(\gamma)}/k[t]} \mathcal{O}_{\overline{B}^{(\gamma)}}, \quad
J_3 := \mathfrak{n}_{r,B} \mathcal{O}_{\overline{B}^{(\gamma)}}. 
\]
For each $i = 1,2,3$, let $W_i$ be the subscheme of $\overline{B}^{(\gamma)}$ defined by $J _i$. 
Note that $\mathfrak{n}_{r,B} \neq 0$. 
Then, by what we have already proved, each $(W_i)_{\infty}$ is a thin set of $\overline{B}^{(\gamma)} _{\infty}$. 
We shall prove the inclusion
\[
W_{\infty} \subset (W_1)_{\infty} \cup (W_2)_{\infty} \cup (W_3)_{\infty}, \tag{i}
\]
which implies that $W_{\infty}$ is also a thin set of $\overline{B}^{(\gamma)} _{\infty}$. 
Suppose that an arc $\alpha \in \overline{B}^{(\gamma)} _{\infty}$ satisfies $\alpha \notin (W_1)_{\infty} \cup (W_2)_{\infty} \cup (W_3)_{\infty}$, that is, 
\begin{enumerate}
\item[(ii)] 
$\operatorname{ord}_{\alpha} \left( \operatorname{Jac}_{\overline{B}^{(\gamma)}/k[t]} \right ) < \infty$, 
\item[(iii)]
$\operatorname{ord}_{\alpha} \left( \operatorname{Jac}_{\widetilde{B}^{(\gamma)}/k[t]} \mathcal{O}_{\overline{B}^{(\gamma)}} \right)< \infty$, and 
\item[(iv)]
$\operatorname{ord}_{\alpha} \left( \mathfrak{n}_{r,B} \mathcal{O}_{\overline{B}^{(\gamma)}} \right) < \infty$. 
\end{enumerate}
By (ii), (iv), and Lemma \ref{lem:L3}, we have
\begin{enumerate}
\item[(v)]
$\operatorname{ord}_{\alpha} \bigl( \operatorname{jac}_{\overline{\mu} _{\gamma}} \bigr) < \infty$. 
\end{enumerate}
Then, by (ii), (iii), (v), and Lemma \ref{lem:additive} (applied to the composite morphism $\overline{\mu}_{\gamma} = \mu_{\gamma} \circ p$), 
we obtain
\begin{enumerate}
\item[(vi)]
$\operatorname{ord}_{\alpha} \left( \operatorname{jac}_p \right) < \infty$. 
\end{enumerate}
Finally, by (ii), (vi), and Lemma \ref{lem:L5}, we have 
\[
\operatorname{ord}_{\alpha} (\mathfrak{n}'_{p}) < \infty, 
\]
which implies $\alpha \notin W_{\infty}$ and finishes the proof of (i).

(3) follows from (2). 
\end{proof}

\begin{lem}\label{lem:D}
Let $\mathfrak{a} \subset \mathcal{O}_B$ be an ideal sheaf on $B$. 
For $\gamma \in G$ and ${\bf b} := (b_1, \ldots, b_7) \in \mathbb{Z}_{\ge 0} ^7$, 
we define a cylinder $D_{\gamma, {\bf b}} \subset \overline{B}^{(\gamma)}_{\infty}$ by
\begin{align*}
D_{\gamma, {\bf b}}
:= 
\operatorname{Cont}^{\ge 1} \left( \mathfrak{m}_x \mathcal{O}_{\overline{B}^{(\gamma)}} \right) 
\cap \operatorname{Cont}^{b_1} \left( \mathfrak{a}\mathcal{O}_{\overline{B}^{(\gamma)}} \right)
\cap \operatorname{Cont}^{b_2} \left( \operatorname{Jac}_{\overline{B}^{(\gamma)}/k[t]} \right) \\
\cap \operatorname{Cont}^{b_3} \left( \operatorname{Jac}_{\widetilde{B}^{(\gamma)}/k[t]} \mathcal{O}_{\overline{B}^{(\gamma)}} \right) 
\cap \operatorname{Cont}^{b_4}  \left( \operatorname{Jac}_{B} \mathcal{O}_{\overline{B}^{(\gamma)}} \right)
\cap \operatorname{Cont}^{b_5} \left( \mathfrak{n}'_{p} \right) \\
\cap \operatorname{Cont}^{b_6} \left( \mathfrak{n}_{r,B} \mathcal{O}_{\overline{B}^{(\gamma)}} \right)
\cap \operatorname{Cont}^{b_7} \left( I_Z \mathcal{O}_{\overline{B}^{(\gamma)}} \right).
\end{align*}
Then, $\overline{\mu} _{\gamma \infty} (D_{\gamma, {\bf b}})$ is a cylinder of $B_{\infty}$, and furthermore, we have 
\begin{align*}
\operatorname{codim}_{B_{\infty}} \bigl( \overline{\mu} _{\gamma \infty} (D_{\gamma, {\bf b}}) \bigr)
= 
\operatorname{codim} _{\overline{B}^{(\gamma)} _{\infty}} (D_{\gamma, {\bf b}})  
+ \operatorname{age}(\gamma) - w_{\gamma}({\bf f})  - b_2  + \frac{b_6}{r}. 
\end{align*}
\end{lem}
\begin{proof}
We fix $\gamma \in G$ and ${\bf b} := (b_1, \ldots, b_7) \in \mathbb{Z}_{\ge 0} ^7$. 

By the definition of $D_{\gamma, {\bf b}}$, we have 
\[
D_{\gamma, {\bf b}} \subset 
\operatorname{Cont} ^{b_2} \left( \operatorname{Jac}_{\overline{B}^{(\gamma)}/k[t]} \right), \quad 
p_{\infty}\left( D_{\gamma, {\bf b}} \right) \subset 
\operatorname{Cont} ^{b_3} \left( \operatorname{Jac}_{\widetilde{B}^{(\gamma)}/k[t]} \right). 
\]
By Lemma \ref{lem:L5}, for any $\alpha \in D_{\gamma, {\bf b}}$, we have 
\[
\operatorname{ord}_{\alpha} (\operatorname{jac}_p) = b_5 - b_2. 
\]
Note that $D_{\gamma, {\bf b}}$ is $C_{\gamma}$-invariant by Lemma \ref{lem:thin}(1). 
Hence, we can apply Proposition \ref{prop:DL2_k[t]} to $p$, and conclude that $p_{\infty}(D_{\gamma, {\bf b}})$ is a cylinder of $\widetilde{B}^{(\gamma)} _{\infty}$, and furthermore, we have 
\[\tag{i}
\operatorname{codim}_{\widetilde{B}^{(\gamma)} _{\infty}} \bigl( p_{\infty}(D_{\gamma, {\bf b}}) \bigr) = 
\operatorname{codim}_{\overline{B}^{(\gamma)} _{\infty}} \left( D_{\gamma, {\bf b}} \right) + b_5 - b_2. 
\]

By the definition of $D_{\gamma, {\bf b}}$, we have 
\[
(\mu _{\gamma} \circ p)_{\infty} \left( D_{\gamma, {\bf b}} \right) \subset 
\operatorname{Cont} ^{b_4} \left( \operatorname{Jac}_{B} \right). 
\]
By Lemma \ref{lem:additive} and Lemma \ref{lem:L3}, 
for any $\alpha \in D_{\gamma, {\bf b}}$, we have 
\begin{align*}
\operatorname{ord}_{p_{\infty} (\alpha)} \left( \operatorname{jac}_{\mu _{\gamma}} \right)
&= \operatorname{ord}_{\alpha} \left( \operatorname{jac}_{\overline{\mu} _{\gamma}} \right) - \operatorname{ord}_{\alpha} \left( \operatorname{jac}_{p} \right)\\
&= \left(\frac{b_6}{r} + \operatorname{age}(\gamma) - w_\gamma({\bf f}) - b_2 \right) - (b_5 - b_2)\\
& = \frac{b_6}{r} - b_5 + \operatorname{age}(\gamma) - w_\gamma({\bf f}). 
\end{align*}
Note that $\mu _{\gamma \infty}$ is injective outside $Z_{\infty}$ by Proposition \ref{prop:DL_X}. 
Since we have $D_{\gamma, {\bf b}} \subset \operatorname{Cont}^{b_7} \left( I_Z \mathcal{O}_{\overline{B}^{(\gamma)}} \right)$ and $b_7 < \infty$, 
we have $(\mu _{\gamma} \circ p)_{\infty}(D_{\gamma, {\bf b}}) \cap Z_{\infty} = \emptyset$. 
Therefore, the restriction map $\mu _{\gamma \infty} |_{p_{\infty} (D_{\gamma, {\bf b}})}$ is injective. 
Hence, we can apply Proposition \ref{prop:EM6.2_k[t]} to $\mu _{\gamma}$, and conclude that $(\mu _{\gamma} \circ p)_{\infty}(D_{\gamma, {\bf b}})$ is a cylinder of $B_{\infty}$, and 
we have 
\begin{align*}\tag{ii}
&\operatorname{codim} _{B_{\infty}} \bigl( (\mu _{\gamma} \circ p)_{\infty}(D_{\gamma, {\bf b}}) \bigr) \\
& = \operatorname{codim} _{\widetilde{B}^{(\gamma)} _{\infty}} \bigl( p_{\infty} (D_{\gamma, {\bf b}}) \bigr) + \frac{b_6}{r} - b_5 + \operatorname{age}(\gamma) - w_\gamma({\bf f}).
\end{align*}
Then, the assertion follows from (i) and (ii). 
\end{proof}

\begin{thm}\label{thm:mld_hyperquot}
Let $\mathfrak{a} \subset  \mathcal{O}_B$ be a non-zero ideal sheaf and $\tau$ a positive real number. 
Then 
\begin{align*}
\operatorname{mld}_x(B,\mathfrak{a}^{\tau}) 
&= \inf _{\gamma \in G, \ b_1, b_2 \in \mathbb{Z}_{\ge 0}}
\left \{ 
\operatorname{codim} \left( C_{\gamma, b_1, b_2} \right) + \operatorname{age}(\gamma)- w_\gamma({\bf f}) - b_2 - \tau b_1
\right \} \\
&= \inf _{\gamma \in G, \ b_1, b_2 \in \mathbb{Z}_{\ge 0}}
\left \{ 
\operatorname{codim} \left( C' _{\gamma, b_1, b_2} \right) + \operatorname{age}(\gamma) - w_\gamma({\bf f}) - b_2 - \tau b_1
\right \}
\end{align*}
holds for 
\begin{align*}
C_{\gamma, b_1, b_2} 
& := 
\operatorname{Cont}^{\ge 1} \left( \mathfrak{m}_x \mathcal{O}_{\overline{B}^{(\gamma)}} \right) \cap
\operatorname{Cont}^{b_1}  \left( \mathfrak{a} \mathcal{O}_{\overline{B}^{(\gamma)}} \right) \cap  
\operatorname{Cont}^{b_2} \left(\operatorname{Jac}_{\overline{B}^{(\gamma)}/k[t]} \right), \\
C'_{\gamma, b_1, b_2}
& := 
\operatorname{Cont}^{\ge 1} \left(\mathfrak{m}_x \mathcal{O}_{\overline{B}^{(\gamma)}} \right) \cap
\operatorname{Cont}^{\ge b_1}  \left(\mathfrak{a} \mathcal{O}_{\overline{B}^{(\gamma)}} \right) \cap 
\operatorname{Cont}^{b_2} \left(\operatorname{Jac}_{\overline{B}^{(\gamma)}/k[t]} \right). 
\end{align*} 
\end{thm}

\begin{proof}
For $\gamma \in G$ and ${\bf b} := (b_1, \ldots, b_7) \in \mathbb{Z}_{\ge 0} ^7$, 
we define the cylinder $D_{\gamma, {\bf b}} \subset \overline{B}^{(\gamma)}_{\infty}$ as in Lemma \ref{lem:D}. 

By \cite{EM09}*{Theorem 7.4}, we have
\[
\operatorname{mld}_x(B,\mathfrak{a}^{\tau})=
\inf _{b_1, b_6 \in \mathbb{Z}_{\ge 0}} 
\left\{\hspace{-1mm}  \begin{array}{c} 
\operatorname{codim} \left( 
\operatorname{Cont}^{\ge 1}(\mathfrak m_x)
\cap \operatorname{Cont}^{b_1} (\mathfrak a) 
\cap \operatorname{Cont}^{b_6}(\mathfrak n_{r,B}) 
\right) \\
-\frac{b_6}{r} - \tau b_1
\end{array}
\right\}.
\]
Here, by Propositions \ref{prop:DL_X}, we have
\begin{align*}
&\operatorname{Cont}^{\ge 1}(\mathfrak m_x) \cap
\operatorname{Cont}^{b_1}(\mathfrak a) \cap 
\operatorname{Cont}^{b_6}(\mathfrak n_{r,B})  \setminus Z_\infty \\
&{}= \bigsqcup _{\langle \gamma \rangle \in \operatorname{Conj}(G)} 
\overline{\mu}_{\gamma \infty} \left( 
\operatorname{Cont}^{\ge 1} \left( \mathfrak{m}_x \mathcal{O}_{\overline{B}^{(\gamma)}} \right) \cap
\operatorname{Cont}^{b_1} \left( \mathfrak{a} \mathcal{O}_{\overline{B}^{(\gamma)}} \right) \cap 
\operatorname{Cont}^{b_6} \left( \mathfrak{n}_{r,B} \mathcal{O}_{\overline{B}^{(\gamma)}} \right) \right) \setminus Z_\infty. 
\end{align*}
Hence, by Proposition \ref{prop:negligible} and Lemma \ref{lem:thin}(3), we have 
\begin{align*}
& \operatorname{codim} \left( \operatorname{Cont}^{\ge 1}(\mathfrak m_x) \cap
\operatorname{Cont}^{b_1}(\mathfrak a) \cap 
\operatorname{Cont}^{b_6}(\mathfrak n_{r,B}) \right) \\
& =
\min_{\gamma \in G,\ b_2, b_3, b_4, b_5, b_7 \in \mathbb{Z}_{\ge 0}} \operatorname{codim} 
\left( \overline{\mu}_{\gamma \infty} ( D_{\gamma, (b_1, \ldots , b_7)} ) \right). 
\end{align*}
 On the other hand, again by Proposition \ref{prop:negligible} and Lemma \ref{lem:thin}(2), we have 
\[
\operatorname{codim} (C_{\gamma, b_1 , b_2}) = 
\min_{b_3, b_4, b_5, b_6, b_7 \in \mathbb{Z}_{\ge 0}} \operatorname{codim} \bigl( D_{\gamma, (b_1, \ldots , b_7)} \bigr).
\]
Therefore, by Lemma \ref{lem:D}, we have
\begin{align*}
&\operatorname{mld}_x(B,\mathfrak{a}^{\tau})\\
&=
\inf _{b_1, b_6} \left\{ 
\operatorname{codim}\left( \operatorname{Cont}^{\ge 1}(\mathfrak m_x)
\cap \operatorname{Cont}^{b_1} (\mathfrak a) 
\cap \operatorname{Cont}^{b_6}(\mathfrak n_{r,B})  \right)
-\frac{b_6}{r}- \tau b_1
 \right\} \\
&= \inf _{\gamma, {\bf b}}
\left\{ 
\operatorname{codim}\bigl( \overline{\mu}_{\gamma \infty} (D_{\gamma, {\bf b}}) \bigr) 
-\frac{b_6}{r}- \tau b_1
 \right\} \\
&= \inf _{\gamma, {\bf b}}
\bigl \{ 
\operatorname{codim}( D_{\gamma, {\bf b}} ) 
+ \operatorname{age}(\gamma) - w_\gamma({\bf f})  - b_2 - \tau b_1
 \bigr \} \\ 
&=\inf _{\gamma, b_1, b_2}
\bigl \{
\operatorname{codim} (C_{\gamma, b_1, b_2}) + \operatorname{age}(\gamma) - w_\gamma({\bf f})  - b_2 - \tau b_1
\bigr \}, 
\end{align*}
which proves the first equality.

The second equality follows from the same arguments as in the proof of \cite{NS22}*{Corollary 4.9}. 
\end{proof}

\begin{rmk}\label{rmk:multi_index}
The formulas in Theorem \ref{thm:mld_hyperquot} can be easily extended to $\mathbb{R}$-ideals.
For an $\mathbb{R}$-ideal sheaf $\mathfrak{a} = \prod _{i=1} ^s \mathfrak{a}_i ^{\tau _i}$ on $B$, 
we have 
\begin{align*}
&\operatorname{mld}_x \left( B, \mathfrak{a} \right) \\
&= 
\inf _{\gamma \in G, w_1, \ldots, w_s, b \in \mathbb{Z}_{\ge 0}}
\left \{ 
\operatorname{codim} \left( C_{\gamma, w_1, \ldots , w_s, b} \right) + \operatorname{age}(\gamma) - w_\gamma({\bf f}) - b - \sum _{i=1} ^s \tau _i w_i
\right \} \\
&=
\inf _{\gamma \in G, w_1, \ldots, w_s, b \in \mathbb{Z}_{\ge 0}}
\left \{ 
\operatorname{codim} \left( C'_{\gamma, w_1, \ldots , w_s, b} \right) + \operatorname{age}(\gamma) - w_\gamma({\bf f}) - b - \sum _{i=1} ^s \tau _i w_i
\right \}
\end{align*}
for 
\begin{align*}
C_{\gamma, w_1, \ldots , w_s, b} & := 
\operatorname{Cont}^{\ge 1} \left( \mathfrak{m}_x \mathcal{O}_{\overline{B}^{(\gamma)}} \right) \cap
\biggl( \bigcap _{i=1} ^{s} \operatorname{Cont}^{w_i} \left(\mathfrak{a}_i \mathcal{O}_{\overline{B}^{(\gamma)}}\right) \biggr) 
 \cap \operatorname{Cont}^{b} \left( \operatorname{Jac}_{\overline{B}^{(\gamma)}/k[t]} \right), \\
C' _{\gamma, w_1, \ldots , w_s, b} & := 
\operatorname{Cont}^{\ge 1} \left( \mathfrak{m}_x \mathcal{O}_{\overline{B}^{(\gamma)}} \right) \cap
\biggl( \bigcap _{i=1} ^{s} \operatorname{Cont}^{\ge w_i}  \left( \mathfrak{a}_i \mathcal{O}_{\overline{B}^{(\gamma)}} \right) \biggr) 
 \cap \operatorname{Cont}^{b} \left( \operatorname{Jac}_{\overline{B}^{(\gamma)}/k[t]} \right). 
\end{align*}
\end{rmk}

\section{PIA formula for hyperquotient singularities}\label{section:PIA}
We keep the notation from Sections \ref{section:setting}, \ref{section:order} and \ref{section:mld}. 
For $1 \le i \le c$, we set 
\[
\overline{H}_i := \operatorname{Spec} \left( R/(f_i) \right) \subset \overline{A}, \qquad 
H_i := \overline{H}_i/G \subset A.
\]

\begin{lem}\label{lem:property B_i}
Let $q:\overline{A} \to A$ denote the quotient map. 
Then the following assertions hold. 
\begin{enumerate}
\item
$\overline{H}_i$ and $H_i$  are prime Weil divisors on some open neighborhood of $\overline{B}$ in $\overline{A}$ and some open neighborhood of $B$ in $A$, respectively.

\item 
We have $d H_i = \operatorname{div} \left( f_i ^d \right)$ on some open neighborhood of $B$ for each $1 \le i \le c$. 
In particular, $H_i$ is a $\mathbb Q$-Cartier divisor on some open neighborhood of $B$, and we have $q^*(H_i)=\overline{H}_i$ on some open neighborhood of $\overline{B}$.

\item 
We have
$(K_A + H_1 + \cdots + H_c) |_B = K_B$. 
\end{enumerate}
\end{lem}
\begin{proof}
Note that $\overline{B}$ is normal by Lemma \ref{lem:etaleness of B}.
For any closed point $x \in \overline{B}$, since $\mathcal O_{\overline{B}, x}$ is a domain, $\mathcal O_{\overline{H}_i, x}$ is also a domain by Lemma \ref{lem:domain} below.
This implies that $(f_i) \mathcal{O}_{\overline{A}, x} \cap \mathcal O_{A, y}$ is a prime ideal for the closed point $y = q(x) \in B$,
which completes the proof of (1). 

Let $x \in \overline{B}$ be a closed point, and let $y := q(x) \in B$. 
We define an ideal $I_{H_i} \subset \mathcal{O}_{A,y}$ by
\[
I_{H_i} := (f_i) \mathcal{O}_{\overline{A}, x} \cap \mathcal{O}_{A,y}. 
\]
By (1), $I_{H_i}$ is a prime ideal. 
Since $f_i$ is $G$-semi-invariant, we have $I_{H_i}^d \subset \left( f_i ^d \right) \subset I_{H_i}$. 
In order to prove $\operatorname{div} \left( f_i ^d \right) = d H_i$ on some open neighborhood of $y$, 
it is sufficient to show 
\[
\bigl( f^d_i \bigr) (\mathcal{O}_{A,y})_{I_{H_i}} =
I_{H_i} ^d (\mathcal{O}_{A,y})_{I_{H_i}}, 
\] 
where $(\mathcal{O}_{A,y})_{I_{H_i}}$ denotes the localization of $\mathcal{O}_{A,y}$ at the prime ideal $I_{H_i}$. 
Since $\overline{A} \to A$ is  \'{e}tale in codimension one, 
the extension of discrete valuation rings 
$(\mathcal{O}_{A,y})_{I_{H_i}} \hookrightarrow (\mathcal{O}_{\overline{A}, x})_{( f_i )}$
is unramified, which implies
\[
( f_i ) (\mathcal{O}_{\overline{A}, x})_{( f_i )} 
= I_{H_i} (\mathcal{O}_{\overline{A}, x})_{( f_i )}. 
\]
Here, $(\mathcal{O}_{\overline{A}, x})_{( f_i )}$ denotes the localization of $\mathcal{O}_{\overline{A}, x}$ at the prime ideal generated by $f_i$. 
Therefore, we have
\begin{align*}
\bigl( f^d_i \bigr) (\mathcal{O}_{A,y})_{I_{H_i}}
&= 
\bigl( f_i \bigr) ^d  (\mathcal{O}_{\overline{A}, x})_{( f_i )} \cap (\mathcal{O}_{A,y})_{I_{H_i}} \\
&= 
I_{H_i} ^d  (\mathcal{O}_{\overline{A}, x})_{( f_i )} \cap (\mathcal{O}_{A,y})_{I_{H_i}} \\
&=
I_{H_i} ^d (\mathcal{O}_{A,y})_{I_{H_i}}, 
\end{align*}
which completes the proof of (2).

Let $x \in B_{\rm sm} \setminus Z$ be a closed point.
Take a closed point $y \in \overline{B}$ whose image in $B$ is $x$.
For $1 \le i \le c$, we define 
\[
\overline{B}_i := \overline{H}_1\cap \cdots \cap \overline{H}_i, \qquad 
B_i := H_1\cap \cdots \cap H_i. 
\]
Then, we claim the following assertions: 
\begin{enumerate}
\item[(i)] 
For each $i$, we have $q^{-1}(H_i) = \overline{H}_i$ at $y$. 

\item[(ii)] 
For each $i$, $\overline{B}_i\to B_i$ is \'{e}tale at $y$. 

\item[(iii)] 
For each $i$, $y$ is a smooth point of $\overline{B}$. 

\item[(iv)]
For each $i$, $y$ is a smooth point of $\overline{B}_i$. 

\item[(v)]
For each $i$, $x$ is a smooth point of $B_i$. 

\item[(vi)]
We have $B_c = B$ at $x$.  
\end{enumerate}
(i) follows from (1) by the same argument as in the proof of Lemma \ref{lem:etaleness of B}(1). 
By (i), we have $q^{-1}(B_i) = \overline{B}_i$ at $y$, which proves (ii). 
(iii) follows since $x$ is a smooth point of $B$ and $\overline{B}\to B$ is \'{e}tale at $y$. 
Since $f_1, \ldots, f_c$ is a regular sequence, (iv) follows from (iii) (cf.\ \cite{Sta}*{tag 00NU}). (v) follows from (ii) and (iv). 
(vi) follows because both $\overline{B}\to B_c$ and $\overline{B}\to B$ are \'{e}tale at $y$, and hence the closed immersion $B \to B_c$ is also \'{e}tale at $y$.

By (v), we have $(K_{B_i}+B_{i+1})|_{B_{i+1}}=K_{B_{i+1}}$ on some open neighborhood of $x$ on $B_{i+1}$.
By (vi) and induction, we obtain $(K_A+H_1+\cdots +H_c)|_B = K_B$ on some open neighborhood of $x$ on $B$.
By moving a closed point $x$ in $B_{\rm sm}\setminus Z$, 
we conclude that $(K_A+H_1+\cdots +H_c)|_B=K_B$ on $B_{\rm sm}\setminus Z$. 
Since $B$ is normal and $\operatorname{codim} _B (B \cap Z) \ge 2$, we have $\operatorname{codim} _B(B_{\rm sing}\cup (B \cap Z))\ge 2$.
Hence, we have $(K_A+H_1+\cdots +H_c)|_B=K_B$ on $B$.
\end{proof}

\begin{lem}\label{lem:domain}
Let $(S,\mathfrak m)$ be a Noetherian local ring and let $f\in S$.
Suppose that $f$ is not a zero divisor of $S$.
If $S/(f)$ is a domain, then $S$ is a domain.
\end{lem}

\begin{proof}
Suppose that $x,y\in S$ satisfy $xy=0$.
Since $xy = 0 \in (f)$ and $(f)$ is a prime ideal, we have $x\in (f)$ or $y\in (f)$. 
Therefore, we can find $x_1 \in S$ or $y_1 \in S$ satisfying $x=fx_1$ or $y=fy_1$.
Since $f$ is not a zero divisor, we have $x_1y=0$ or $xy_1=0$.
Repeating the same argument, we have $x\in \bigcap _{m\in\mathbb N} (f)^m$ or $y\in \bigcap_{m\in\mathbb N}(f)^m$.
Since $\bigcap_{m\in\mathbb N} (f)^m\subset\bigcap_{m\in\mathbb N} \mathfrak m^m=0$ by \cite{Mat89}*{Theorem 8.9}, 
we conclude that $x=0$ or $y=0$, which implies that $S$ is a domain.
\end{proof}

\begin{lem}\label{lem:zure}
Let $\mathfrak{a} \subset \mathcal{O}_A$ be a non-zero ideal sheaf on $A$. 
Let $\tau$ be a non-negative real number. 
Let $x \in A$ be the origin. 
Then, we have
\begin{align*}
&\operatorname{mld}_x \left( A, (f_1 ^d \cdots f_c ^d )^{\frac{1}{d}} \mathfrak{a} ^{\tau} \right) \\
&= \inf _{\gamma \in G, b_1, b_2 \in \mathbb{Z}_{\ge 0}}
\left \{ 
\operatorname{codim} (C_{\gamma, b_1, b_2}) + \operatorname{age}(\gamma) - w_{\gamma} ({\bf f}) - b_1 - \tau b_2 
\right \},
\end{align*}
where $C_{\gamma, b_1, b_2} \subset \overline{A}^{(\gamma)} _{\infty}$ is defined by
\begin{align*}
C_{\gamma, b_1, b_2} := {}
&\operatorname{Cont}^{\ge 1} \left(\mathfrak{m}_x \mathcal{O}_{\overline{A}^{(\gamma)}} \right)
\cap
\operatorname{Cont}^{\ge b_2} \left( \mathfrak{a} \mathcal{O}_{\overline{A}^{(\gamma)}} \right) \\
&\cap 
\operatorname{Cont}^{\ge b_1} \left( \lambda ^* _{\gamma}(f_1)t^{-w_\gamma(f_1)}\cdots \lambda ^* _{\gamma}(f_c)t^{-w_\gamma(f_c)} \right). 
\end{align*} 
\end{lem}
\begin{proof}
By \cite{NS22}*{Corollary 4.9}, we have 
\begin{align*}\tag{i}
&\operatorname{mld}_x \left( A, (f_1 ^d \cdots f_c ^d )^{\frac{1}{d}} \mathfrak{a} ^{\tau} \right)\\
&= \inf _{\gamma \in G, b_0, b_2 \in \mathbb{Z}_{\ge 0}}
\left \{ 
\operatorname{codim} 
\left (D_{\gamma, b_0, b_2} \right) + \operatorname{age}(\gamma) -\frac{b_0}{d} - \tau b_2
\right \},
\end{align*}
where $D_{\gamma, b_0, b_2} \subset \overline{A}^{(\gamma)} _{\infty}$ denotes
\[
D_{\gamma, b_0, b_2} := 
\operatorname{Cont}^{\ge 1} \left( \mathfrak{m}_x \mathcal{O}_{\overline{A}^{(\gamma)}} \right) 
\cap
\operatorname{Cont}^{\ge b_2} \left( \mathfrak{a} \mathcal{O}_{\overline{A}^{(\gamma)}} \right) 
\cap
\operatorname{Cont}^{\ge b_0} \left( \lambda^* _{\gamma} \left( f_1 ^d \cdots f_c ^d \right) \right). 
\]

We fix $\gamma \in G$ and $b_2 \in \mathbb{Z}_{\ge 0}$. 
Note that 
\[
\lambda^* _{\gamma} \left( f_1 ^d \cdots f_c ^d \right) = 
t^{d w_{\gamma}({\bf f})} \left( \lambda ^* _{\gamma} (f_1) t^{- w_{\gamma} (f_1)} \cdots \lambda ^* _{\gamma} (f_c) t^{- w_{\gamma} (f_c)} \right)^d. 
\]
Therefore, for any $b_0 \in \mathbb{Z}_{\ge 0}$, if we define $b' _0 := \max \left\{ \left \lceil \frac{b_0}{d} - w_{\gamma}({\bf f}) \right \rceil, 0 \right\}$, we have 
\[
\operatorname{Cont}^{\ge b_0} \left( \lambda^* _{\gamma} \left( f_1 ^d \cdots f_c ^d \right) \right)
= \operatorname{Cont}^{\ge b' _0} \left( \lambda ^* _{\gamma} (f_1) t^{- w_{\gamma} (f_1)} \cdots \lambda ^* _{\gamma} (f_c) t^{- w_{\gamma} (f_c)} \right). 
\]
Thus, we have 
\begin{align*}
\operatorname{codim} \left( D_{\gamma, b_0, b_2} \right) - \frac{b_0}{d} 
&= \operatorname{codim}\left( C_{\gamma, b'_0, b_2} \right) - \frac{b_0}{d} \\
&\ge \operatorname{codim}\left( C_{\gamma, b'_0, b_2} \right) - b'_0 - w_{\gamma}({\bf f}). 
\end{align*}
On the other hand, for any $b_1 \in \mathbb{Z}_{\ge 0}$, 
if we define $b' _1 := d b_1 + d w_{\gamma}({\bf f})$, we have 
\[
\operatorname{Cont}^{\ge b' _1} \left( \lambda^* _{\gamma} \left( f_1 ^d \cdots f_c ^d \right) \right)
= \operatorname{Cont}^{\ge b_1} \left( \lambda ^* _{\gamma} (f_1) t^{- w_{\gamma} (f_1)} \cdots \lambda ^* _{\gamma} (f_c) t^{- w_{\gamma} (f_c)} \right). 
\]
Thus, we have 
\[
\operatorname{codim} \left( C_{\gamma, b_1, b_2} \right) - b_1 - w_{\gamma}({\bf f}) 
= 
\operatorname{codim} \left( D_{\gamma, b'_1, b_2} \right) - \frac{b'_1}{d}. 
\]
Therefore, the assertion follows from (i). 
\end{proof}

\begin{thm}\label{thm:PIA}
Let $\mathfrak{a} \subset \mathcal{O}_A$ be a non-zero ideal, 
and let $\tau$ be a non-negative real number. 
Suppose that $\mathfrak{a} \mathcal{O}_B \not = 0$. 
Let $x \in A$ be the origin. 
Suppose that $B$ is klt. 
Then, we have
\[
\operatorname{mld}_x \left(A, H_1 + \cdots + H_c, \mathfrak{a} ^{\tau} \right)
=\operatorname{mld}_x \left( B, \mathfrak{a}^{\tau} \mathcal{O}_B \right). 
\]
\end{thm}

\begin{proof}
By Lemma \ref{lem:property B_i}(3), it is easy to see 
\[
\operatorname{mld}_x \left( A,H_1 + \cdots + H_c, \mathfrak{a} ^{\tau} \right) 
\le \operatorname{mld}_x \left( B, \mathfrak{a}^{\tau} \mathcal{O}_B \right).
\]
In what follows, we shall prove the opposite inequality. 

By Lemma \ref{lem:property B_i}(2), we have 
\[\tag{i}
\operatorname{mld}_x \left( A,H_1 + \cdots + H_c, \mathfrak{a} ^{\tau} \right) 
=
\operatorname{mld}_x \left( A, (f_1 ^d \cdots f_c ^d )^{\frac{1}{d}} \mathfrak{a} ^{\tau} \right). 
\]
By Lemma \ref{lem:zure}, we have
\begin{align*}\tag{ii}
&\operatorname{mld}_x \left( A, (f_1 ^d \cdots f_c ^d )^{\frac{1}{d}} \mathfrak{a} ^{\tau} \right) \\
&= \inf _{\gamma \in G, b_1, b_2 \in \mathbb{Z}_{\ge 0}}
\left \{ 
\operatorname{codim}_{\overline{A}^{(\gamma)}_\infty} (C_{\gamma, b_1, b_2}) + \operatorname{age}(\gamma) - w_{\gamma} ({\bf f}) - b_1 - \tau b_2 
\right \},
\end{align*}
where $C_{\gamma, b_1, b_2} \subset \overline{A}^{(\gamma)} _{\infty}$ is defined by
\begin{align*}
C_{\gamma, b_1, b_2} := {}
&\operatorname{Cont}^{\ge 1} \left(\mathfrak{m}_x \mathcal{O}_{\overline{A}^{(\gamma)}} \right)
\cap
\operatorname{Cont}^{\ge b_2} \left( \mathfrak{a} \mathcal{O}_{\overline{A}^{(\gamma)}} \right) \\
&\cap 
\operatorname{Cont}^{\ge b_1} \left( \lambda ^* _{\gamma}(f_1)t^{-w_\gamma(f_1)}\cdots \lambda ^* _{\gamma}(f_c)t^{-w_\gamma(f_c)} \right). 
\end{align*}
On the other hand, by Theorem \ref{thm:mld_hyperquot}, we have 
\begin{align*}\tag{iii}
&\operatorname{mld}_x \left( B, \mathfrak{a}^{\tau} \mathcal{O}_B \right)\\
&= \inf _{\gamma \in G, b_2, b_3 \in \mathbb{Z}_{\ge 0}}
\left \{ 
\operatorname{codim}_{\overline{B}^{(\gamma)}_\infty} \left( D_{\gamma, b_2, b_3} \right) + \operatorname{age}(\gamma)- w_\gamma({\bf f}) - b_3 - \tau b_2
\right \}, 
\end{align*}
where $D_{\gamma, b_2, b_3} \subset \overline{B}^{(\gamma)} _{\infty}$ denotes
\[
D_{\gamma, b_2, b_3} 
:= 
\operatorname{Cont}^{\ge 1} \left(\mathfrak{m}_x \mathcal{O}_{\overline{B}^{(\gamma)}} \right) \cap
\operatorname{Cont}^{\ge b_2}  \left(\mathfrak{a} \mathcal{O}_{\overline{B}^{(\gamma)}} \right) \cap 
\operatorname{Cont}^{b_3} \left(\operatorname{Jac}_{\overline{B}^{(\gamma)}/k[t]} \right). 
\]

We fix $\gamma \in G$, and $b_1, b_2 \in \mathbb{Z}_{\ge 0}$. 
By (i)-(iii), it is sufficient to find $b_3 \in \mathbb{Z}_{\ge 0}$ satisfying
\[\tag{iv}
\operatorname{codim}_{\overline{A}^{(\gamma)}_\infty} \left( C_{\gamma, b_1, b_2} \right) - b_1
\ge
\operatorname{codim}_{\overline{B}^{(\gamma)}_\infty} \left( D_{\gamma, b_2, b_3} \right) - b_3. 
\]
Take an irreducible component $C'_{\gamma, b_1, b_2} \subset C_{\gamma, b_1, b_2}$ satisfying 
\[\tag{v}
\operatorname{codim}_{\overline{A}^{(\gamma)}_\infty} \left( C_{\gamma, b_1, b_2} \right ) = \operatorname{codim}_{\overline{A}^{(\gamma)}_\infty} \left( C' _{\gamma, b_1, b_2} \right). 
\]
Set 
\[
b_3 := \min _{\alpha \in C'_{\gamma, b_1, b_2} \cap \overline{B}^{(\gamma)}_{\infty}} \operatorname{ord}_{\alpha} \left( \operatorname{Jac}_{\overline{B}^{(\gamma)}/k[t]} \right). 
\]
First, by \cite{NS22}*{Claim 5.2}, we have $b_3 < \infty$. 
Here, we use the assumption that $B$ is klt (see Remark \cite{NS22}*{Remark 5.3}). 
We set 
\[
C''_{\gamma, b_1, b_2} := C'_{\gamma, b_1, b_2} \cap \operatorname{Cont}^{\le b_3} 
\left( (i^*) ^{-1} \operatorname{Jac}_{\overline{B}^{(\gamma)}/k[t]} \right), 
\]
where $i$ is the inclusion $i : \overline{B}^{(\gamma)} \hookrightarrow \overline{A}^{(\gamma)}$. 
Since $C''_{\gamma, b_1, b_2}$ is a non-empty open subcylinder of an irreducible closed cylinder $C'_{\gamma , b_1, b_2}$, 
we have 
\[\tag{vi}
\operatorname{codim}_{\overline{A}^{(\gamma)}_\infty} \left( C'' _{\gamma, b_1, b_2} \right)
= 
\operatorname{codim}_{\overline{A}^{(\gamma)}_\infty} \left( C' _{\gamma, b_1, b_2} \right). 
\]
Next, by applying Lemma \ref{lem:EM8.4} to $C''_{\gamma, b_1, b_2}$ and $\lambda _{\gamma}(f_1)t^{-w_\gamma(f_1)}, \ldots, \lambda _{\gamma}(f_c)t^{-w_\gamma(f_c)}$, we have 
\[\tag{vii}
\operatorname{codim}_{\overline{A}^{(\gamma)}_\infty} 
\left( 
C''_{\gamma, b_1, b_2}
\right) 
+ b_3 - b_1
\ge
\operatorname{codim}_{\overline{B}^{(\gamma)}_{\infty}} \left(
C''_{\gamma, b_1, b_2} \cap \overline{B}^{(\gamma)}_{\infty}
\right).
\]
Since we have 
\[
C''_{\gamma, b_1, b_2} \cap \overline{B}^{(\gamma)}_{\infty}
= C'_{\gamma, b_1, b_2} \cap \overline{B}^{(\gamma)}_{\infty} \cap 
\operatorname{Cont}^{b_3} 
\left( \operatorname{Jac}_{\overline{B}^{(\gamma)}/k[t]} \right)
\subset D_{\gamma, b_2, b_3}, 
\]
we have 
\[\tag{viii}
\operatorname{codim}_{\overline{B}^{(\gamma)}_{\infty}} \left(
C''_{\gamma, b_1, b_2} \cap \overline{B}^{(\gamma)}_{\infty}
\right)
\ge 
\operatorname{codim}_{\overline{B}^{(\gamma)}_{\infty}} 
\left( D_{\gamma, b_2, b_3} \right). 
\]
By (v)-(viii), we conclude that this $b_3$ satisfies (iv). We complete the proof. 
\end{proof}

\begin{rmk}\label{rmk:multi_index2}
The formulas in Theorem \ref{thm:PIA} can be extended to $\mathbb{R}$-ideals due to Remark \ref{rmk:multi_index}. 
Let $\mathfrak{a}$ be an $\mathbb{R}$-ideal on $A$. 
Then we have
\[
\operatorname{mld}_x \bigl(A, H_1 + \cdots + H_c, \mathfrak{a} \bigr)
=\operatorname{mld}_x (B, \mathfrak{a}\mathcal O_{B}). 
\]
\end{rmk}

\section{PIA conjecture and LSC conjecture}\label{section:LSC}
We summarize our main theorem (Theorem \ref{thm:PIA}) with a little generalization. 
Note that Theorem \ref{thm:PIA} is nothing but 
Theorem \ref{thm:PIA2} for the case when $y$ is the image $x$ of the origin of $\mathbb{A}^N _k$. 

\begin{thm}\label{thm:PIA2}
Let $G \subset {\rm GL}_N(k)$ be a finite subgroup which does not contain a pseudo-reflection. 
Let $X := \mathbb{A}_k^N / G$ be the quotient variety. 
Let $Z \subset X$ be the minimum closed subset such that $\mathbb{A}_k^N \to X$ is \'{e}tale outside $Z$. 
Let $\overline{Y}$ be a subvariety of $\mathbb{A}_k^N$ of codimension $c$. 
Suppose that $\overline{Y} \subset \mathbb{A}_k^N$ is defined by $c$ $G$-semi-invariant equations 
$f_1, \ldots , f_c \in k[x_1, \ldots, x_N]$. 
Let $Y := \overline{Y}/ G$ be the quotient variety. 
We assume that $Y$ has only klt singularities and $\operatorname{codim} _Y (Y \cap Z) \ge 2$. 
Then, for any $\mathbb{R}$-ideal sheaf $\mathfrak{a}$ on $X$ with $\mathfrak{a} \mathcal{O}_Y \not = 0$, and any closed point $y \in Y$, we have 
\[
\operatorname{mld}_y \bigl( Y, \mathfrak{a}\mathcal{O}_Y \bigr) = 
\operatorname{mld}_y \bigl( X, (f_1 ^d  \cdots f_{c} ^d )^{\frac{1}{d}} \mathfrak{a} \bigr). 
\]
\end{thm}
\begin{proof}
We follow the argument in the proof of \cite{NS22}*{Theorem 6.2}. 
We fix a closed point $y \in Y$. 
Take a closed point $y' \in \mathbb{A}_k^N$ whose image in $Y$ is $y$. 
Let $G_{y'} := \{ g \in G \mid g(y') = y'\}$ be the stabilizer group of $y'$ and let $Z' \subset \mathbb{A}_k ^N / G_{y'}$ be the minimum closed subset such that $\mathbb{A}_k^N \to \mathbb{A}_k ^N / G_{y'}$ is \'{e}tale outside $Z'$. 
Then $\mathbb{A}_k ^N / G_{y'} \to \mathbb{A}_k ^N / G$ is \'{e}tale at $y$ and $\operatorname{codim} _{\overline{Y}/G_{y'}} \left( \left( \overline{Y}/G_{y'} \right) \cap Z' \right) \ge 2$. 
Note that the minimal log discrepancy is preserved under an \'{e}tale map. 
Hence by replacing $X = \mathbb{A}_k ^N / G$ by $\mathbb{A}_k ^N / G_{y'}$ and changing the coordinate of $\mathbb{A}_k ^N$, 
we may assume that $y'$ is the origin and the group action is still linear. 
Then we have 
\[
\operatorname{mld}_y \bigl( Y, \mathfrak{a}\mathcal{O}_Y \bigr) = 
\operatorname{mld}_y \bigl( X, (f_1^d \cdots f_{c}^d)^{\frac{1}{d}} \mathfrak{a} \bigr). 
\]
by Theorem \ref{thm:PIA}, which completes the proof. 
\end{proof}

The PIA conjecture holds for quotient singularities $X$ and a prime Weil divisor $D$ which is a quotient of a $G$-invariant divisor. Note that $D$ may not be Cartier.
\begin{cor}\label{cor:PIA4}
Let $G$, $X$ and $Z$ be the group and the schemes as in the statement of Theorem \ref{thm:PIA2}. 
Let $\overline{D}$ be a $G$-invariant Cartier divisor of $\mathbb{A}_k^N$. 
Let $D := \overline{D}/ G$ be the quotient variety. 
We assume that $D$ has only klt singularities and $\operatorname{codim} _D (D \cap Z) \ge 2$. 
Then, for any $\mathbb{R}$-ideal sheaf $\mathfrak{a}$ on $X$ with $\mathfrak{a} \mathcal{O}_D \not = 0$, and any closed point $y \in D$, we have 
\[
\operatorname{mld}_y \bigl( D, \mathfrak{a}\mathcal{O}_D \bigr) = 
\operatorname{mld}_y \bigl( X, D, \mathfrak{a} \bigr). 
\]
\end{cor}
\begin{proof}
This theorem follows from the case $c = 1$ in Theorem \ref{thm:PIA2}.
\end{proof}

As a corollary of Theorem \ref{thm:PIA2} (Theorem \ref{thm:PIA}), we can prove the PIA conjecture for the $Y$ and its klt Cartier divisor $D$. 

\begin{thm}\label{thm:PIA3}
Let $Y$ and $Z$ be the schemes as in the statement of Theorem \ref{thm:PIA2}. 
Let $D$ be a Cartier divisor on $Y$. 
Suppose that $\operatorname{codim}_D (D \cap Z) \ge 2$ and that $D$ is klt at a closed point $y \in D$. 
Then, for any $\mathbb{R}$-ideal sheaf $\mathfrak{a}$ on $Y$ with $\mathfrak{a} \mathcal{O}_D \not = 0$, we have 
\[
\operatorname{mld}_y \left( Y,D, \mathfrak{a} \right) = 
\operatorname{mld}_y \left( D, \mathfrak{a} \mathcal{O}_D \right). 
\]
\end{thm}
\begin{proof}
By the same argument as in the proof of Theorem \ref{thm:PIA2}, we may assume that $y = x$, where $x$ is the image of the origin of $\mathbb{A}^N _k$. 
Take $g \in \mathcal{O}_{X, x}$ such that its image 
$\overline{g} \in \mathcal{O}_{Y, x}$ defines $D$ at $x$. 
Take an $\mathbb{R}$-ideal sheaf $\mathfrak{b}$ on $X$ satisfying $\mathfrak{a} = \mathfrak{b} \mathcal{O}_Y$.
Then, by applying Theorem \ref{thm:PIA} twice (cf.\ \cite{NS22}*{Remark 5.5}), we have
\begin{align*}
\operatorname{mld}_{x} \left( Y,D, \mathfrak{a} \right) 
&= 
\operatorname{mld}_{x} \left( Y, (g \cdot \mathfrak{b}) \mathcal{O}_Y  \right) \\
&=
\operatorname{mld}_{x} \left( X, (f_1^d \cdots f_{c}^d \cdot g^d)^{\frac{1}{d}}\mathfrak{b} \right) \\
&=
\operatorname{mld}_{x} ( D, \mathfrak{a} \mathcal{O}_D ), 
\end{align*}
which completes the proof. 
\end{proof}

As a corollary of Theorem \ref{thm:PIA2}, we can prove the lower semi-continuity of minimal log discrepancies for the same $Y$. 

\begin{thm}\label{thm:LSC_general}
Let $Y$ be the scheme as in the statement of Theorem \ref{thm:PIA2}. 
Then, for any $\mathbb{R}$-ideal sheaf $\mathfrak{b}$ on $Y$ the function 
\[
|Y| \to \mathbb{R}_{\ge 0} \cup \{ - \infty \}; \quad y \mapsto \operatorname{mld}_y(Y,\mathfrak{b})
\]
is lower semi-continuous, where we denote by $|Y|$ the set of all closed points of $Y$ with the Zariski topology. 
\end{thm}
\begin{proof}
We use the same notations as in Theorem \ref{thm:PIA2}. 
Take an $\mathbb{R}$-ideal sheaf $\mathfrak{a}$ on $X$ satisfying $\mathfrak{a} \mathcal{O}_Y = \mathfrak{b}$. Then, by Theorem \ref{thm:PIA2}, for any closed point $y \in Y$, we have 
\[
\operatorname{mld}_y \bigl( Y, \mathfrak{b} \bigr) = 
\operatorname{mld}_y \bigl( X, (f_1 ^d \cdots f_{c} ^d)^{\frac{1}{d}} \mathfrak{a} \bigr). 
\]
Since the lower semi-continuity is known for $X$ by \cite{Nak16}*{Corollary 1.3}, 
we can conclude the lower semi-continuity for $Y$. 
\end{proof}

\end{document}